\documentclass[12pt,a4paper,reqno]{amsart}
\usepackage[mathcal]{eucal}
\usepackage{amssymb}
\usepackage[nice]{nicefrac}
\usepackage{hyperref}
\usepackage[normalem]{ulem}
\usepackage{color}

\theoremstyle{plain}
\newtheorem{theorem}{Theorem}[section]
\newtheorem{lemma}[theorem]{Lemma}
\newtheorem{proposition}[theorem]{Proposition}
\newtheorem{corollary}[theorem]{Corollary}

\theoremstyle{definition}

\newtheorem{problem}[theorem]{Problem}

\newtheorem{remark}[theorem]{Remark}

\numberwithin{equation}{section}

\DeclareMathOperator{\hdim}{hdim}
\DeclareMathOperator{\hspec}{hspec}

\newcommand{\N}{{\ensuremath \mathbb{N}}}

\begin{document}

\title[Infinite normal Hausdorff spectra]{A pro-$p$ group with
  infinite normal \\ Hausdorff spectra}

\author[B. Klopsch]{Benjamin Klopsch} \address{Benjamin Klopsch:
  Mathematisches Institut, Heinrich-Heine-Universit\"at, 40225
  D\"usseldorf, Germany} \email{klopsch@math.uni-duesseldorf.de}

\author[A. Thillaisundaram]{Anitha Thillaisundaram} 
\address{Anitha Thillaisundaram: School of Mathematics and Physics,
  University of Lincoln, Lincoln LN6 7TS, England}
\email{anitha.t@cantab.net}

\date{\today}

\thanks{The second author acknowledges the support from the Alexander
  von Humboldt Foundation and from the Forscher-Alumni-Programm of the Heinrich-Heine-Universit\"at D\"usseldorf (HHU); she thanks HHU for its hospitality.}

\keywords{pro-$p$ groups, Hausdorff dimension, 
    Hausdorff spectrum, normal Hausdorff spectrum}

\subjclass[2010]{Primary 20E18 ; Secondary 28A78}


\begin{abstract} 
  Using wreath products, we construct a finitely generated pro-$p$
  group $G$ with infinite normal Hausdorff spectrum
  \[
  \hspec_\trianglelefteq^\mathcal{P}(G) = \{ \hdim_G^\mathcal{P}(H)
  \mid H \trianglelefteq_\mathrm{c} G \};
  \]
  here
  $\hdim_G^\mathcal{P} \colon \{ X \mid X \subseteq G \} \to [0,1]$
  denotes the Hausdorff dimension function associated to the $p$-power
  series $\mathcal{P} \colon G^{p^i}$, $i \in \mathbb{N}_0$.  More
  precisely, we show that
  $\hspec_\trianglelefteq^\mathcal{P}(G) = [0,\nicefrac{1}{3}] \cup
  \{1\}$
  contains an infinite interval; this settles a question of Shalev.  
  Furthermore, we prove that the normal Hausdorff spectra
  $\hspec_\trianglelefteq^\mathcal{S}(G)$ with respect to other
  filtration series~$\mathcal{S}$ have a similar shape.  In
  particular, our analysis applies to standard filtration series such
  as the Frattini series, the lower $p$-series and the modular
  dimension subgroup series.

  Lastly, we pin down the ordinary Hausdorff spectra
  $\hspec^\mathcal{S}(G) = \{ \hdim_G^\mathcal{S}(H) \mid H
  \le_\mathrm{c} G \}$
  with respect to the standard filtration series~$\mathcal{S}$.  The
  spectrum $\hspec^\mathcal{L}(G)$ for the lower
  $p$-series~$\mathcal{L}$ displays surprising new features.
\end{abstract}

\maketitle


\section{Introduction}

The concept of Hausdorff dimension has led to interesting applications
in the context of profinite groups; see \cite{KlThZu19} and the
references given therein.  Let $G$ be a countably based infinite
profinite group and consider a \emph{filtration series} $\mathcal{S}$
of~$G$, that is, a descending chain
$G = G_0 \supseteq G_1 \supseteq \ldots$ of open normal subgroups
$G_i \trianglelefteq_\mathrm{o} G$ such that $\bigcap_i G_i = 1$.  These open
normal subgroups form a base of neighbourhoods of the identity and
induce a translation-invariant metric on~$G$ given by
$d^\mathcal{S}(x,y) = \inf \left\{ \lvert G : G_i \rvert^{-1} \mid x
  \equiv y \pmod{G_i} \right\}$,
for $x,y \in G$.  This, in turn, supplies the \emph{Hausdorff
  dimension} $\hdim_G^\mathcal{S}(U) \in [0,1]$ of any subset
$U \subseteq G$, with respect to the filtration series~$\mathcal{S}$.

Barnea and Shalev~\cite{BaSh97} established the following
`group-theoretic' interpretation of the Hausdorff dimension of a
closed subgroup $H$ of $G$ as a logarithmic density:
\begin{equation*}
  \hdim_G^\mathcal{S}(H) =
  \varliminf_{i\to \infty} \frac{\log \lvert HG_i : G_i
    \rvert}{\log \lvert G : G_i \rvert}.
\end{equation*}

The \emph{Hausdorff spectrum} of $G$, with respect to $\mathcal{S}$,
is
\[
\hspec^\mathcal{S}(G) = \{ \hdim_G^\mathcal{S}(H) \mid H \le_\mathrm{c} G\}
\subseteq [0,1],
\]
where $H$ runs through all closed subgroups of~$G$.  As indicated by
Shalev in~\cite[\S 4.7]{Sh00}, it is also natural to consider the
\emph{normal Hausdorff spectrum} of~$G$, with respect to
$\mathcal{S}$, namely
\[
\hspec^{\mathcal{S}}_{\trianglelefteq}(G) = \{ \hdim^{\mathcal{S}}_G(H) \mid H
\trianglelefteq_\mathrm{c} G \}
\]
which reflects the range of Hausdorff dimensions of closed normal
subgroups.  Apart from the observations in~\cite[\S 4.7]{Sh00}, very
little appears to be known about normal Hausdorff spectra of profinite
groups.

Throughout we will be concerned with pro-$p$ groups, where $p$ denotes
an odd prime; in Appendix~\ref{sec:case-p-2} we indicate how our
results extend to~$p=2$.  We recall that even for well structured
groups, such as $p$-adic analytic pro-$p$ groups~$G$, the Hausdorff
dimension function and the Hausdorff spectrum of $G$ are known to be
sensitive to the choice of~$\mathcal{S}$; compare~\cite{KlThZu19}.
However, for a finitely generated pro-$p$ group $G$ there are natural
choices for $\mathcal{S}$, such as the $p$-power series~$\mathcal{P}$,
the Frattini series~$\mathcal{F}$, the lower $p$-series~$\mathcal{L}$
and the modular dimension subgroup series~$\mathcal{D}$; see
Section~\ref{sec:prelim}.

In this paper, we are interested in a particular group $G$ constructed
as follows.  The pro\nobreakdash-$p$ wreath product
$W = C_p \mathrel{\hat{\wr}} \mathbb{Z}_p$ is the inverse limit
$\varprojlim_{k \in \N} C_p \mathrel{\wr} C_{p^k}$ of the finite
standard wreath products of cyclic groups with respect to the natural
projections; clearly, $W$ is $2$-generated as a topological group.
Let $F$ be the free pro-$p$ group on two generators and let
$R \trianglelefteq_\mathrm{c} F$ be the kernel of a presentation
$\pi \colon F \to W$.  We are interested in the pro-$p$ group
\[
G = F/N, \qquad \text{where} \quad N = [R,F] R^p
\trianglelefteq_\mathrm{c} F.
\]
Up to isomorphism, the group $G$ does not depend on the particular
choice of~$\pi$, as can be verified using Gasch\"utz' Lemma;
see~\cite[Prop.~2.2]{Lu01}.  Indeed, $G$ can be described as the
universal covering group for $2$-generated central extensions of
elementary abelian pro-$p$ groups by~$W$, i.e., for $2$-generated
pro-$p$ groups $E$ admitting a central elementary abelian subgroup $A$
such that $E/A \cong W$.

\begin{theorem} \label{thm:main-thm} For $p>2$, the normal Hausdorff
  spectra of the pro-$p$ group $G$ constructed above, with respect to
  the standard filtration series $\mathcal{P}$, $\mathcal{D}$,
  $\mathcal{F}$ and $\mathcal{L}$ respectively, satisfy:
  \begin{align*}
    \hspec^{\mathcal{P}}_{\trianglelefteq}(G) &=
                                                \hspec^{\mathcal{D}}_{\trianglelefteq}(G)
                                                = [0,\nicefrac{1}{3}]
                                                \cup \{1\}, \\ 
    \hspec^{\mathcal{F}}_{\trianglelefteq}(G) &=
                                                [0,\nicefrac{1}{(1+p)}]
                                                \cup \{1\}, \\
    \hspec^{\mathcal{L}}_{\trianglelefteq}(G) & 
                                                = [0,\nicefrac{1}{5}]
                                                \cup
                                                \{\nicefrac{3}{5}\}\cup
                                                \{1\}. 
  \end{align*}
  In particular, they each contain an infinite real interval.
\end{theorem}

This solves a problem posed by Shalev~\cite[Problem~16]{Sh00}.
We observe that the normal Hausdorff spectrum of $G$ is sensitive to changes
in filtration and that the normal Hausdorff spectrum of $G$ with
respect to the Frattini series varies with~$p$.

In Section~\ref{sec:general-description} we show that finite direct
powers $G \times \ldots \times G$ of the group $G$ provide examples of
normal Hausdorff spectra consisting of multiple intervals.
Furthermore, the sequence $G \times \overset{m} \ldots \times G$,
$m \in \mathbb{N}$, has normal Hausdorff spectra `converging' to
$[0,1]$; compare Corollary~\ref{cor:mult-intervals}.  We highlight
three natural problems.

\begin{problem}
  Does there exist a finitely generated pro-$p$ group $H$
  \begin{enumerate}
  \item[(a)] with countably infinite normal Hausdorff spectrum
    $\hspec_{\trianglelefteq}^\mathcal{S}(H)$,
  \item[(b)] with full normal Hausdorff spectrum
    $\hspec_{\trianglelefteq}^\mathcal{S}(H) = [0,1]$,
  \item[(c)] such that $1$ is not an isolated point in
    $\hspec_{\trianglelefteq}^\mathcal{S}(H)$,
  \end{enumerate}
  for one or several of the standard series
  $\mathcal{S} \in \{ \mathcal{P}, \mathcal{D}, \mathcal{F},
  \mathcal{L} \}$?
\end{problem}

We also compute the entire Hausdorff spectra of $G$ with respect to
the four standard filtration series, answering en route a question
raised in~\cite[VIII.7.2]{Kl99}.

\begin{theorem} \label{thm:entire-spectrum} For $p>2$, the Hausdorff
  spectra of the pro-$p$ group $G$ constructed above, with respect to
  the standard filtration series, satisfy:
  \begin{align*}
    \hspec^{\mathcal{P}}(G) & = \hspec^{\mathcal{D}}(G) =
                              \hspec^{\mathcal{F}}(G) = [0,1],\\
    \hspec^{\mathcal{L}}(G) & =  [0,\nicefrac{4}{5}) \cup
                              \{\nicefrac{3}{5} + \nicefrac{2m}{5p^n}\mid
                              m, n \in\mathbb{N}_0 \text{ with }
                              \nicefrac{p^n}{2} < m\le p^n\}.
  \end{align*}
\end{theorem}

The qualitative shape of the spectrum $\hspec^{\mathcal{L}}(G)$, i.e.,
its decomposition into a continuous and a non-continuous, but dense
part, is unprecedented and of considerable interest; in
Corollary~\ref{cor:spectrum-W-L} we show that already the wreath
product $W = C_p \mathrel{\hat{\wr}} \mathbb{Z}_p$ has a similar
Hausdorff spectrum with respect to the lower $p$-series.

\medskip

\noindent \emph{Organisation}.  Section~\ref{sec:prelim} contains
preliminary results.  In Section~\ref{sec:presentation} we give an
explicit presentation of the pro-$p$ group $G$ and describe a series
of finite quotients $G_k$, $k \in \mathbb{N}$, such that
$G = \varprojlim G_k$.  In Section~\ref{sec:general-description} we
provide a general description of the normal Hausdorff spectrum of $G$
and, with respect to certain induced filtration series, we generalise
this to finite direct powers of~$G$.  In
Section~\ref{sec:p-power-series} we compute the normal Hausdorff
spectrum of $G$ with respect to the $p$-power series~$\mathcal{P}$,
and in Section~\ref{sec:other-series} we compute the normal Hausdorff
spectra of $G$ with respect to the other three standard filtration
series~$\mathcal{D}, \mathcal{F}, \mathcal{L}$.  In
Section~\ref{sec:entire-spectrum} we compute the entire Hausdorff
spectra of~$G$.  Finally, in Appendix~\ref{sec:case-p-2} we indicate
how our results extend to the case $p=2$.

\medskip

\noindent \textit{Notation.}  Throughout, $p$ denotes an \emph{odd}
prime, although some results hold also for $p=2$, possibly with minor
modifications; only in Appendix~\ref{sec:case-p-2}
we discuss the analogous pro-$2$ groups.  We denote by
$\varliminf_{i \to \infty} a_i$ the lower limit (limes inferior) of a
sequence $(a_i)_{i \in \mathbb{N}}$ in
$\mathbb{R} \cup \{ \pm \infty \}$.  Tacitly, subgroups of profinite
groups are generally understood to be closed subgroups.  Subscripts
are used to emphasise that a subgroup is closed respectively open, as
in $H \le_\mathrm{c} G$ respectively $H \le_\mathrm{o} G$.  We use
left-normed commutators, e.g., $[x,y,z] = [[x,y],z]$.

\medskip

\noindent \textbf{Acknowledgement.} Discussions between Amaia
Zugadi-Reizabal and the first author several years ago led to the
initial idea that the pro-$p$ group $G$ constructed in this paper
should have an infinite normal Hausdorff spectrum with respect to the
$p$-power filtration series.  At the time Zugadi-Reizabal was
supported by the Spanish Government, grant MTM2011-28229-C02-02,
partly with FEDER funds, and by the Basque Government, grant
IT-460-10.  We also thank Yiftach Barnea, Dan Segal and Matteo
Vannacci for useful conversations, and the referee for suggesting
several improvements to the exposition.


\section{Preliminaries} \label{sec:prelim} 

\subsection{} Let $G$ be a finitely generated pro-$p$ group.  We
consider four natural filtration series on $G$.  The \emph{$p$-power
  series} of $G$ is given by
\[
\mathcal{P} \colon G^{p^i} = \langle x^{p^i} \mid x \in G
\rangle, \quad i \in \mathbb{N}_0.
\]
The \emph{lower $p$-series} (or lower $p$-central
series) of $G$ is given recursively by
\begin{align*}
  \mathcal{L} \colon P_1(G) = G,  %
  & \quad \text{and} \quad  P_i(G) = P_{i-1}(G)^p
    \, [P_{i-1}(G),G] \quad
    \text{for $i \geq 2$,} \\
  \intertext{while the \emph{Frattini series} of $G$ is
  given recursively by} 
  \mathcal{F} \colon \Phi_0(G) = G, %
  & \quad \text{and} \quad \Phi_i(G) = \Phi_{i-1}(G)^p
    \, [\Phi_{i-1}(G),\Phi_{i-1}(G)] \quad \text{for $i \geq 1$.} \\
  \intertext{The (modular) \emph{dimension subgroup series} (or
  Jennings series or Zassenhaus series) of $G$ can be defined recursively
  by}
  \mathcal{D} \colon D_1(G) = G, %
  & \quad \text{and} \quad D_i(G) = D_{\lceil i/p \rceil}(G)^p 
    \prod_{1 \le j <i} [D_j(G),D_{i-j}(G)] \quad \text{for $i \geq 2$.}  
\end{align*}
As a default we set $P_0(G) = D_0(G) = G$.


\subsection{} Next, we collect auxiliary results to detect Hausdorff
dimensions of closed subgroups of pro-$p$ groups.  For a countably
based infinite pro-$p$ group $G$, equipped with a filtration series
$\mathcal{S} \colon G = G_0 \supseteq G_1\supseteq \ldots$, and a
closed subgroup $H \le_\mathrm{c} G$ we say that $H$ has \emph{strong
  Hausdorff dimension in $G$ with respect to $\mathcal{S}$} if
\[
\hdim^{\mathcal{S}}_G(H) = \lim_{i \to \infty} \frac{\log_p \lvert H
  G_i : G_i \rvert}{\log_p \lvert G : G_i \rvert}
\]
is given by a proper limit.

The first lemma is an easy variation of \cite[Lem.~5.3]{KlThZu19} and
we omit the proof.

\begin{lemma} \label{lem:interval} Let $G$ be a countably based
  infinite pro-$p$ group with closed subgroups
  $K \le_\mathrm{c} H \le_\mathrm{c} G$.  Let
  $\mathcal{S} \colon G = G_0 \supseteq G_1\supseteq \ldots$ be a
  filtration series of~$G$ and write
  $\mathcal{S} \vert_H \colon H = H_0\supseteq H_1 \supseteq \ldots $,
  with $H_i =H\cap G_i$ for $i \in \mathbb{N}_0$, for the induced
  filtration series of~$H$.  If $K$ has strong Hausdorff dimension
  in~$H$ with respect to $\mathcal{S} \vert_H $, then
  \[
  \hdim^{\mathcal{S}}_{G}(K) = \hdim^{\mathcal{S}}_{G}(H) \cdot
  \hdim^{\mathcal{S} \vert_H}_H(K).
  \]
\end{lemma}

\begin{lemma} \label{lem:hdims-from-ses} Let $G$ be a countably based
  infinite pro-$p$ group with closed subgroups
  $N \trianglelefteq_\mathrm{c} G$ and $H \le_\mathrm{c} G$.  Let
  $\mathcal{S} \colon G = G_0 \supseteq G_1\supseteq \ldots$ be a
  filtration series of~$G$, and consider the induced filtration series
  of~$N$ and~$G/N$ defined by
  \[
  \mathcal{S} \vert_N \colon G_i \cap N, \; i \in \mathbb{N}_0, \qquad
  \text{and} \qquad \mathcal{S} \vert_{G/N} \colon G_iN/N, \; i \in
  \mathbb{N}_0.
  \]
  Suppose that $N$ has strong Hausdorff dimension
  $\xi = \hdim_G^\mathcal{S}(N)$ in~$G$, with respect
  to~$\mathcal{S}$.  Then we have
  \begin{align}
    \hdim_G^\mathcal{S}(H) & \ge (1-\xi) \, \hdim_{G/N}^{\mathcal{S}
                             \vert_{G/N}}(HN/N) + \xi \,
                             \varliminf_{i \to \infty} \frac{\log_p
                             \lvert HG_i \cap N : G_i \cap N \rvert}{\log_p
                             \lvert N : G_i \cap N \rvert}
                             \tag{$\ast$} \label{equ:inequ-equ-1} \\
                           & \ge (1-\xi) \, \hdim_{G/N}^{\mathcal{S}
                             \vert_{G/N}}(HN/N) + \xi \,
                             \hdim_N^{\mathcal{S} \vert_N}(H \cap
                             N). \tag{$\ast\ast$} \label{equ:inequ-equ-2}
  \end{align}
  Moreover, equality holds in~\eqref{equ:inequ-equ-1}, if $HN/N$ has
  strong Hausdorff dimension in $G/N$ with respect to
  $\mathcal{S} \vert_{G/N}$ or if the lower limit 
    on the right-hand side
  is actually a limit.  Similarly, equality holds
  in~\eqref{equ:inequ-equ-2} if
   \begin{enumerate}
   \item[(i)] $H \cap N \le_\mathrm{o} N$ is an open subgroup or
   \item[(ii)] $G_i N = (G_i \cap H)N$, for all sufficiently large
     $i \in \mathbb{N}$.
   \end{enumerate}
\end{lemma}

\begin{proof}
  We observe that
  \begin{align*}
    \hdim^{\mathcal{S}}_G(H) %
    & = \varliminf_{i \to \infty}  \Bigg( \; \underbrace{\frac{\log_p \lvert G : N G_i
      \rvert}{\log_p \lvert G : G_i \rvert}}_{\to\, 1- \xi \text{ as } i
      \to \infty} \frac{\log_p \lvert HG_iN : G_i N
      \rvert}{\log_p \lvert G : G_i N \rvert} \\
    & \qquad \qquad + \underbrace{\frac{\log_p \lvert NG_i : G_i
      \rvert}{\log_p \lvert G : G_i \rvert}}_{\to\, \xi \text{ as } i
      \to \infty} 
      \frac{\log_p \lvert HG_i \cap NG_i : G_i
      \rvert}{\log_p \lvert N G_i : G_i \rvert} \; \Bigg)
  \end{align*}
  and that, for each $i \in \mathbb{N}_0$,
  \[
  \frac{\log_p \lvert HG_i \cap NG_i : G_i \rvert}{\log_p \lvert NG_i
    : G_i \rvert} = \frac{\log_p \lvert HG_i \cap N : G_i \cap N
    \rvert}{\log_p \lvert N : G_i \cap N \rvert}.
  \]
 Finally,
 \[
 \log_p \lvert HG_i \cap N : G_i \cap N \rvert \ge \log_p \lvert (H
 \cap N)(G_i \cap N) : G_i \cap N \rvert
 \]
 and, if condition (i) or (ii) holds, the difference between the two
 terms is bounded by a constant that is independent of
 $i \in \mathbb{N}_0$.
\end{proof}

\begin{lemma} \label{lem:hdim-elem-ab} Let $Z \cong C_p^{\, \aleph_0}$
  be a countably based infinite elementary abelian pro-$p$ group,
  equipped with a filtration series~$\mathcal{S}$.  Then, for every
  $\eta \in [0,1]$, there exists a closed subgroup
  $K \le_\mathrm{c} Z$ with strong Hausdorff dimension $\eta$ in $Z$
  with respect to~$\mathcal{S}$.
\end{lemma}

\begin{proof}
  Write
  $\mathcal{S} \colon Z = Z_0 \supseteq Z_1 \supseteq Z_2 \supseteq
  \ldots$
  and let $\eta \in [0,1]$.  For $i \in \mathbb{N}$, we have
  $Z_{i-1}/Z_i \cong C_p^{\,d_i}$ for non-negative integers~$d_i$.

  \smallskip

  \noindent \underline{Claim}: There exist non-negative integers
  $e_1,e_2, \ldots$ such that, for each $i \in \mathbb{N}$, we have
  $0\le e_i\le d_i$ and
  \[
  e_1+\ldots + e_i=\lceil \eta (d_1+\ldots +d_i)\rceil.
  \]

  \smallskip

  Indeed, with $e_1 = \lceil \eta d_1\rceil$ the statement holds true
  for $i=1$.  Now, let $i \ge 2$ and suppose that
  $e_1+\ldots +e_{i-1} = \lceil \eta(d_1+\ldots +d_{i-1})\rceil$.
  Then
  \[
  \lceil \eta(d_1+\ldots +d_{i-1})\rceil \le \lceil \eta(d_1+\ldots
  +d_{i})\rceil \le \lceil \eta(d_1+\ldots +d_{i-1})\rceil +d_i
  \]
  and thus we may set
  \[
  e_i=\lceil \eta(d_1+\ldots +d_{i})\rceil-(e_1+\ldots +e_{i-1}),
  \]
  to satisfy the statement for~$i$.  The claim is proved.

  \smallskip

  For all sufficiently large $i \in \mathbb{N}$ we have
  $d_1+\ldots +d_i > 0$ and
  \[
  \eta \le \frac{e_1+\ldots +e_i}{d_1+\ldots+d_i} \le \eta +
  \frac{1}{d_1+\ldots+d_i}.
  \]

  With these preparations, it suffices to display a subgroup
  $K \le_\mathrm{c} Z$ such that
  \[
  \log_p \vert K Z_ i :Z_i \rvert = e_1+\ldots+e_i.
  \]
  For this purpose, we write
  \[
  Z = \langle z_{1,1}, \ldots, z_{1,d_1}, \;\; z_{2,1}, \ldots,
  z_{2,d_2},\;\; \ldots\; , \;\; z_{i,1}, \ldots, z_{i,d_i}, \;\; \ldots
  \rangle
  \]
  such that $Z_{i-1} = \langle z_{i,1}, \ldots, z_{i,d_i} \rangle Z_i$
  for each $i \in \mathbb{N}$.  Then we set
  \[
  K = \langle z_{1,1}, \ldots, z_{1,e_1}, \;\; z_{2,1}, \ldots,
  z_{2,e_2}, \;\; \ldots\; , \; z_{i,1},\ldots, z_{i,e_i},\;\; \ldots
  \rangle. \qedhere
  \]
\end{proof}

\begin{corollary} \label{cor:abelian-section-interval} Let $G$ be a
  countably based pro-$p$ group, equipped with a filtration
  series~$\mathcal{S}$, and let
  $N \trianglelefteq_\mathrm{c} H \le_\mathrm{c} G$ such that
  $H/N \cong C_p^{\, \aleph_0}$.  Set $\xi = \hdim_G^\mathcal{S}(N)$
  and $\eta = \hdim_G^\mathcal{S}(H)$.  If $N$ or $H$ has strong
  Hausdorff dimension in~$G$ with respect to~$\mathcal{S}$, then
  $[\xi,\eta] \subseteq \hspec^\mathcal{S}(G)$.
\end{corollary}

\begin{proof}
  If $N$ has strong Hausdorff dimension, we apply
  Lemmata~\ref{lem:interval}, \ref{lem:hdims-from-ses}
  and~\ref{lem:hdim-elem-ab}.  If $H$ has strong Hausdorff dimension
  the claim follows from~\cite[Thm.~5.4]{KlThZu19}.
\end{proof} 


\subsection{} For convenience we recall two standard commutator
collection formulae.

\begin{proposition}\label{pro:standard-commutator-id}
  Let $G = \langle a,b \rangle$ be a finite $p$-group, and let
  $r \in \N$.  For $u,v \in G$ let $K(u,v)$ denote the normal closure
  in $G$ of \textup{(i)} all commutators in $\{u,v\}$ of weight at
  least $p^r$ that have weight at least $2$ in $v$, together with
  \textup{(ii)} the $p^{r-s+1}$th powers of all commutators in
  $\{u,v\}$ of weight less than $p^s$ and of weight at least $2$
  in~$v$ for $1 \le s \leq r$.  Then
  \begin{align}
    (ab)^{p^r} & \equiv_{K(a,b)} a^{p^r} \, b^{p^r} \, [b, a]^{\binom{p^r}{2}} \,
             [b, a, a]^{\binom{p^r}{3}} \, \cdots \,
             [b, a, \overset{p^r-2} \ldots, a]^{\binom{p^r}{p^r-1}} \,
             [b, a, \overset{p^r-1} \ldots,
             a], \label{equ:commutator-formula-1} \\ 
    [a^{p^r}, b] & \equiv_{K(a,[a,b])} [a, b]^{p^r} \, [a,b,a]^{\binom{p^r}{2}} \,
               \cdots \, [a, b, a, \overset{p^r-2} \ldots, a]^{\binom{p^r}{p^r-1}}
               \, [a, b, a, \overset{p^r-1} \ldots, a]. \label{equ:commutator-formula-2} 
  \end{align}

  \noindent \textbf{Remark.} Under the standing assumption $p \ge 3$ and the
  extra assumptions
  \[
  \gamma_2(G)^p = 1 \qquad \text{and} \qquad [\gamma_2(G),\gamma_2(G)]
  \subseteq Z(G),
  \]
  the congruences \eqref{equ:commutator-formula-1} and
  \eqref{equ:commutator-formula-2} simplify to
  \begin{equation} \label{equ:commutator-formula-3}
  (ab)^{p^r} \equiv_{L(a,b)} a^{p^r} \, b^{p^r} \, [b, a, \overset{p^r-1}
  \ldots, a] \quad \text{and} \quad [a^{p^r}, b] \equiv_{M(a,b)} \, [a, b, a,
  \overset{p^r-1} \ldots, a],
  \end{equation}
  where $L(a,b)$ denotes the normal closure in $G$ of all commutators
  in $\{a,b\}$ of weight at least $p^r$ that have weight at least $2$
  in~$b$ and $M(a,b)$ denotes the normal closure in $G$ of all
  commutators
  $[[b,a,\overset{i}{\ldots},a],[b,a,\overset{j}{\ldots},a]]$ with
  $i+j \ge p^r$.
\end{proposition}

The general result is recorded (in a slighter stronger form) in
\cite[Prop.~1.1.32]{LeMc02}; we remark that
\eqref{equ:commutator-formula-2} follows directly from
\eqref{equ:commutator-formula-1}, due to the identity
$[a^{p^r},b] = a^{-p^r} (a [a,b])^{p^r}$.  The first congruence
in~\eqref{equ:commutator-formula-3} follows directly
from~\eqref{equ:commutator-formula-1}; the second congruence
in~\eqref{equ:commutator-formula-3} is derived
from~\eqref{equ:commutator-formula-2} by standard
commutator manipulations.


\subsection{} Now we describe, for $k\in \mathbb{N}$, the lower
central series, the lower $p$-series and the Frattini series of the
finite wreath product
\[
W_k = \langle x,y \rangle = \langle x \rangle \ltimes \langle y, y^x,
\ldots, y^{x^{p^k-1}} \rangle \cong C_p \mathrel{\wr} C_{p^k}
\]
with top group $\langle x \rangle \cong C_{p^k}$ and base group
$\langle y, y^x, \ldots, y^{x^{p^k-1}} \rangle \cong C_p^{\, p^k}$.
 
\begin{proposition} \label{pro:series-Wk} For $k \in \mathbb{N}$, the
  finite wreath product $W_k$ defined above is nilpotent of
  class~$p^k$ and
  $W_k^{\, p^k} = \langle y y^x y^{x^2} \cdots y^{x^{p^k-1}} \rangle
  \cong C_p$.
  \begin{enumerate}
  \item[(1)] The lower central series of $W_k$ satisfies
    \begin{align*}
      & W_k = \gamma_1(W_k)  = \langle x,y \rangle \; \gamma_2(W_k)%
      & \text{with} \quad %
      & W_k/\gamma_2(W_k) \cong C_{p^k} \times C_p, \\
      & \gamma_i(W_k) = \langle [y,x,\overset{i-1}{\ldots},x] \rangle
        \; \gamma_{i+1}(W_k) %
      & \text{with} \quad %
      & \gamma_i(W_k)/\gamma_{i+1}(W_k) \cong C_p \text{ for $2 \le i
        \le p^k$.}
    \end{align*} 
  \item[(2)] The lower $p$-series of $W_k$ has length $p^k$; it
    satisfies, for $1 \le i \le k$,
    \begin{align*}
      P_i(W_k) & = \langle x^{p^{i-1}}, [y,x,\overset{i-1}{\ldots},x]
                 \rangle \; P_{i+1}(W_k)%
      & \text{with} \quad%
      &  P_i(W_k)/P_{i+1}(W_k) \cong C_p \times C_p,\\
      \intertext{and, for $k < i \le p^k$,}
      P_i(W_k) & = \langle [y,x,\overset{i-1}{\ldots},x] \rangle \;
                 P_{i+1}(W_k)%
      &  \text{with} \quad%
      &  P_i(W_k)/P_{i+1}(W_k) \cong C_p.
    \end{align*}
  \item[(3)] The Frattini series of $W_k$ has length $k+1$; it
    satisfies, for $0\le i< k$,
    \begin{align*}
      \Phi_i(W_k) & =\langle x^{p^i} \rangle
                    \gamma_{\frac{p^i-1}{p-1}+1}(W_k) \quad 
                    \text{with}\quad \Phi_i(W_k)/\Phi_{i+1}(W_k)\cong
                    C_p \times \overset{p^i+1}{\ldots} \times C_p,\\ 
      \Phi_k(W_k) & = \gamma_{\frac{p^k-1}{p-1}+1}(W_k) \qquad\quad
                    \text{with}\quad\Phi_k(W_k)/\Phi_{k+1}(W_k)\cong
                    C_p \times
                    \overset{\frac{p^{k+1}-2p^k+1}{p-1}}{\ldots}  
                    \times C_p. 
    \end{align*} 
    \item[(4)] The dimension subgroup series of $W_k$ has
      length~$p^k$; in particular, it satisfies,
      for $p^{k-1}+1 \le i \le p^k$,
    \begin{align*}
      & D_i(W_k) = \gamma_i(W_k)
        = \langle [y,x,\overset{i-1}{\ldots},x] \rangle
        \; D_{i+1}(W_k) %
      & \text{with} \quad %
      & D_i(W_k)/D_{i+1}(W_k) \cong C_p.
    \end{align*} 
  \end{enumerate}
\end{proposition}

\begin{proof}
  The assertions are well known and easy to verify from the concrete
  realisation of $W_k$ as a semidirect product
  \begin{equation}\label{equ:Wk-explicit}
    W_k \cong \langle 1+t \rangle / \langle (1+t)^{p^k} \rangle \ltimes
    \mathbb{F}_p[t] / t^{p^k}
    \mathbb{F}_p[t]
  \end{equation}
  in terms of polynomials over the finite field~$\mathbb{F}_p$: here
  $y^{x^i}$ corresponds to $(1+t)^{i}$ modulo
  $t^{p^k} \mathbb{F}_p[t]$, and it is easy to describe all normal
  subgroups.  In particular the normal subgroups of $W_k$ contained in
  the base group form a descending chain, corresponding to the groups
  $t^{i-1} \mathbb{F}_p[t] / t^{p^k} \mathbb{F}_p[t]$,
  $1 \le i \le p^k+1$.

  For $0 \le m < k$ and
  $z \in \langle y, y^x, \ldots, y^{x^{p^k-1}} \rangle$ the element
  \[
  (x^{p^m} z)^{p^k} = (x^{p^m})^{p^k} z^{x^{(p^k-1)p^m}} \cdots
  z^{x^{p^m}} z = z^{x^{(p^k-1)p^m}} \cdots z^{x^{p^m}} z
  \]
  corresponds in $\mathbb{F}_p[t] / t^{p^k} \mathbb{F}_p[t]$ to a
  multiple of
  \[
  \sum_{i=0}^{p^k-1} (1+t)^{i p^m} = \sum_{i=0}^{p^k-1} (1+t^{p^m})^i
  = \frac{(1+t^{p^m})^{p^k}-1}{(1+t^{p^m}) -1} = t^{(p^k-1)p^m};
  \]
  this shows that
  $W_k^{\, p^k} = \langle y y^x y^{x^2} \cdots y^{x^{p^k-1}} \rangle
  \cong C_p$.

  Clearly, $\gamma_1(W_k) = W_k$.  For $2 \le i \le p^k+1$, the group
  $\gamma_i(W_k)$ corresponds to the subgroup
  $t^{i-1} \mathbb{F}_p[t] / t^{p^k} \mathbb{F}_p[t]$ of the base
  group.  In particular, $W_k$ has nilpotency class~$p^k$.  For
  $1 \le i \le k$, we have
  $P_i(W_k) = \langle x^{p^{i-1}} \rangle \gamma_i(W_k)$, while for
  $k < i \le p^k$ we get $P_i(W_k) = \gamma_i(W_k)$.  For
  $0 \le i \le k$ a simple induction shows that the group
  $\Phi_i(W_k)$ is the normal closure in $W_k$ of the two elements
  \[
  x^{p^i} \qquad \text{and} \qquad [y,x,x^p,x^{p^2},\ldots,
  x^{p^{i-1}}]
  =[y,x,\overset{\frac{p^i-1}{p-1}}{\ldots},x];
  \]
  the intersection of $\Phi_i(W_k)$ with the base group corresponds to
  $t^{\frac{p^i-1}{p-1}} \mathbb{F}_p[t] / t^{p^k} \mathbb{F}_p[t]$.
  Thus
  $\Phi_i(W_k) = \langle x^{p^i} \rangle \gamma_{\frac{p^i-1}{p-1}
    +1}(W_k)$.
  In particular, $\Phi_k(W_k)$ is elementary abelian and
  $\Phi_i(W_k) =1$ for $i>k$.  Finally, for $i \ge p^{k-1} + 1$, we
  use~\cite[Thm.~11.2]{DDMS99} to deduce that
  $D_i(W_k) = \gamma_i(W_k)$.
\end{proof}

The structural results for the finite wreath products $W_k$ transfer
naturally to the inverse limit~$W \cong \varprojlim_k W_k$, i.e., the
pro-$p$ wreath product
\begin{equation}\label{equ:pro-p-wreath-prod-W}
  W = \langle x,y \rangle = \langle x \rangle \ltimes B \cong C_p
  \mathrel{\hat{\wr}} \mathbb{Z}_p
\end{equation}
with top group $\langle x \rangle \cong \mathbb{Z}_p$ and base group
$B = \prod_{i \in \mathbb{Z}} \langle y^{x^i} \rangle \cong C_p^{\,
  \aleph_0}$.
Compatible with~\eqref{equ:Wk-explicit}, the group $W$ has a concrete
realisation as a semidirect product
\begin{equation} \label{equ:W-explicit}
  W \cong \langle 1+t \rangle \ltimes
  \mathbb{F}_p[\![t]\!] , \quad \text{where $\langle 1+t \rangle
    \le_\mathrm{c} \mathbb{F}_p[\![t]\!]^*$,}
\end{equation}
in terms of formal power series over the finite field~$\mathbb{F}_p$.
We record the following lemma on closed normal subgroups of $W$.

\begin{lemma} \label{lem:normal-subgroups-in-W} Let
  $W = \langle x,y \rangle \cong C_p \mathrel{\hat{\wr}} \mathbb{Z}_p$
  with base group~$B$ as above, and let
  $1 \ne K \trianglelefteq_\mathrm{c} W$ be a non-trivial closed
  normal subgroup.  Then either $K$ is open in $W$ or $K$ is open
  in~$B$; in particular, $K \cap B \le_\mathrm{o} B$ and
  $\lvert K \cap B : [K \cap B,W] \rvert = p$.
\end{lemma}

\begin{proof}
  The lower central series of $W$ is well known and easy to compute:
  $\gamma_1(W) = W$ and $\gamma_i(W) = B_{i-1}$ for $i \ge 2$, where
  $B = B_0 \ge B_1 \ge B_2 \ge \ldots$ with
  $B_{i-1} = \langle [y,x,\overset{i-1}{\ldots},x] \rangle B_i$ and
  $\lvert B_{i-1} : B_i \rvert = p$; in other words,
  $\langle x \rangle$ acts uniserially on~$B$; compare
  Proposition~\ref{pro:series-Wk}.

  It follows that $1 \neq K \cap B = B_i$ for some non-negative
  integer~$i$, hence $K \cap B \le_\mathrm{o} B$ and
  $\lvert K \cap B : [K \cap B,W] \rvert = \lvert B_i : B_{i+1} \rvert
  = p$.
  Suppose now that $K \not \subseteq B$.  Then there exists
  $x^mz \in K$ with $m \in \mathbb{N}$ and $z \in B$.  We may assume
  that $m = p^k$ is a $p$-power.  Then
  $\langle x^{p^k}z \rangle \ltimes B \cong \langle x \rangle \ltimes
  (B \times \overset{p^k}{\ldots} \times B)$,
  where $x^{p^k}z$ maps to $x$ and, on the right-hand side, $x$ acts
  diagonally and in each coordinate according to the original action
  in~$W$.  Hence we may assume that $x \in K$.  Now the description of
  the lower central series of $W$ yields
  $\langle x \rangle B_1 \le_\mathrm{c} K$ and thus
  $K \le_\mathrm{o} W$.
\end{proof}

  From Proposition~\ref{pro:series-Wk} and
  Lemma~\ref{lem:normal-subgroups-in-W} we deduce the following;
  cf.~\cite[Ch.~VIII.7]{Kl99}.

\begin{corollary} \label{cor:normal-hspec-W}
  The normal Hausdorff spectrum of the pro-$p$ group
  $W = C_p \mathrel{\hat{\wr}} \mathbb{Z}_p$ with respect to the
  standard filtration series $\mathcal{P}$, $\mathcal{D}$,
  $\mathcal{F}$ and $\mathcal{L}$ respectively, satisfies:
  \[
  \hspec^{\mathcal{P}}_{\trianglelefteq}(W) =
  \hspec^{\mathcal{D}}_{\trianglelefteq}(W) =
  \hspec^{\mathcal{F}}_{\trianglelefteq}(W)
  = \{0,1\} \quad \text{and} \quad
  \hspec^{\mathcal{L}}_{\trianglelefteq}(W) = \{0,\nicefrac{1}{2}, 1\}.
  \]
\end{corollary}

The next result is well known (and not difficult to prove directly);
compare~\cite[Cor.~12.5.10]{Wi98}.  It gives a first indication that
Theorem~\ref{thm:main-thm} is at least plausible.

\begin{proposition} \label{pro:infinitely-presented}
 The pro-$p$ group $W = C_p \mathrel{\hat{\wr}} \mathbb{Z}_p$ 
 is not finitely presented. 
\end{proposition}

The final result in this section concerns the \emph{finitely generated
  Hausdorff spectrum} of the pro-$p$
group~$W = C_p \mathrel{\hat{\wr}} \mathbb{Z}_p$, with respect to a
standard filtration series $\mathcal{S}$; it is defined as
\[
\hspec^{\mathcal{S}}_{\text{fg}}(W) = \{ \hdim^{\mathcal{S}}_W(H) \mid
H \le_\mathrm{c} W \text{ and } H \text{ finitely generated} \}
\]
and reflects the range of Hausdorff dimensions of (topologically)
finitely generated subgroups; compare~\cite[\S 4.7]{Sh00}.

\begin{theorem}\label{thm:finitely-generated-spectrum}
  With respect to the standard filtration series $\mathcal{P}$,
  $\mathcal{D}$, $\mathcal{F}$ and $\mathcal{L}$ respectively, the
  pro-$p$ group $W = C_p \mathrel{\hat{\wr}} \mathbb{Z}_p$ satisfies:
  \begin{align*}
    \hspec^{\mathcal{P}}_{\mathrm{fg}}(W) %
    & = \hspec^{\mathcal{D}}_{\mathrm{fg}}(W) = \hspec^{\mathcal{F}}_{\mathrm{fg}}(W) 
      = \{\nicefrac{m}{p^n}\mid n\in\mathbb{N}_0, 0 \le m\le p^n\}, \\
    \hspec^{\mathcal{L}}_{\mathrm{fg}}(W) %
    & = \{0\} \cup  \{\nicefrac{1}{2}+\nicefrac{m}{2p^n}\mid
      n\in\mathbb{N}_0, 0 \le m \le p^n\}.
  \end{align*}
\end{theorem}

\begin{proof}
  As above, let $B$ denote the base group of the wreath
  product~$W = \langle x,y \rangle$.  Let
  $\mathcal{S} \in \{ \mathcal{P}, \mathcal{D}, \mathcal{F},
  \mathcal{L} \}$,
  and let $K$ be a finitely generated subgroup of~$W$.

  If $K \subseteq B$ then $K$ is finite and
  $\hdim^{\mathcal{S}}_W(K) = 0$.  Now suppose that
  $K \not \subseteq B$; in the proof below we will no longer use that
  $K$ is finitely generated, but it will become clear that this is
  automatically so.  Write $K = \langle x^{p^n} z \rangle \ltimes M$,
  where $n \in \mathbb{N}_0$, $z \in B$ and $M = K \cap B$.  Let
  $B = B_0 \ge B_1 \ge \ldots$ be the filtration corresponding to
  $\mathbb{F}_p[\![t]\!] \ge t \mathbb{F}_p[\![t]\!] \ge \ldots$ under
  ~\eqref{equ:W-explicit}, as in the proof of
  Lemma~\ref{lem:normal-subgroups-in-W}.  We set
  \[
  J = \{ j \in \mathbb{N}_0 \mid (M \cap B_j) \not\subseteq B_{j+1} \}
  \quad \text{and} \quad J_0 = \{ j + p^n\mathbb{Z} \mid j \in J \}
  \subseteq \mathbb{Z}/p^n\mathbb{Z}.
  \]
  Under the isomorphism~\eqref{equ:W-explicit}, we may regard $M$ as
  an $\mathbb{F}_p[\![t^{p^n}]\!]$-submodule of
  $\mathbb{F}_p[\![t]\!]$.  Hence $J + p^n \mathbb{N}_0 \subseteq J$
  and
  \[
  \lim_{i \to \infty} \frac{\log_p \lvert (K \cap B) B_i : B_i \rvert}{\log_p \lvert B : B_i
    \rvert} = \frac{\lvert J_0 \rvert}{p^n}.
  \]  
  From Proposition~\ref{pro:series-Wk} it is easily seen that 
  $B$ has strong Hausdorff dimension 
  \[
  \hdim^{\mathcal{P}}_W(B) =  \hdim^{\mathcal{D}}_W(B) =
  \hdim^{\mathcal{F}}_W(B) = 1 \quad \text{and} \quad
  \hdim^{\mathcal{L}}_W(B) = \nicefrac{1}{2};
  \]
  compare Corollary~\ref{cor:normal-hspec-W}.  Using
  Lemma~\ref{lem:hdims-from-ses}, we deduce that
  \[
 \hdim^{\mathcal{S}}_W(K) = (1 - \hdim^{\mathcal{S}}_W(B)) +
 \hdim^{\mathcal{S}}_W(B) \frac{\lvert J_0 \rvert}{p^n}
  \]
  lies in the desired range; in fact, the argument even shows that $K$
  has strong Hausdorff dimension.

  Conversely, our analysis above shows that, for $n \in \mathbb{N}_0$
  and $0 \le m \le p^n$, the subgroup
  $K_{n,m} = \langle x^{p^n}, [y,x], [y,x,x], \ldots,
  [y,x,\overset{m}{\ldots},x] \rangle$ has Hausdorff dimension
 \[
 \hdim^{\mathcal{S}}_W(K_{n,m}) = 
 \begin{cases}
   \nicefrac{m}{p^n} & \text{if $\mathcal{S} \in \{ \mathcal{P},
     \mathcal{D}, \mathcal{F} \}$,} \\
   \nicefrac{1}{2} + \nicefrac{m}{2p^n} & \text{if
     $\mathcal{S} = \mathcal{L}$.}
 \end{cases} \qedhere
 \]
\end{proof}

The next corollary answers a question raised
in~\cite[VIII.7.2]{Kl99}; it was shown there that
$[0,\nicefrac{1}{2}]\subseteq \hspec^{\mathcal{L}}(W)$, while
$(\nicefrac{1}{2},1) \, \cap\,
\hspec^{\mathcal{L}}(W)$ remained undetermined.

\begin{corollary}\label{cor:spectrum-W-L}
  The Hausdorff spectrum of the pro-$p$ group
  $W = C_p \mathrel{\hat{\wr}} \mathbb{Z}_p$ with respect to the lower
  $p$-series $\mathcal{L}$ is
  \[
  \hspec^{\mathcal{L}}(W) = [0,\nicefrac{1}{2}] \cup
  \big\{\nicefrac{1}{2}+\nicefrac{m}{2p^n} \mid n \in \mathbb{N}_0, 1
  \le m \le p^n-1 \big\} \cup \{1\}.
 \]
 Furthermore, every subgroup $K \le_\mathrm{c} W$ with $\hdim^\mathcal{L}_W(K)
 > \nicefrac{1}{2}$ has strong Hausdorff dimension in~$W$, with
 respect to~$\mathcal{L}$.
\end{corollary}

\begin{proof}
  The subgroups contained in the base group $B$ of $W$ yield
  $[0,\nicefrac{1}{2}]$ as part of the Hausdorff spectrum;
  cf.~Lemma~\ref{lem:hdim-elem-ab}.  The proof of
  Theorem~\ref{thm:finitely-generated-spectrum} shows that the
  subgroups not contained in $B$ yield the remaining part of the
  claimed spectrum and that each of them has strong Hausdorff
  dimension in~$W$.
\end{proof}


\section{An explicit presentation for the pro-$p$ group~$G$ and \\ a
  description of its finite quotients $G_k$ for $k \in
  \mathbb{N}$} \label{sec:presentation}

Recall that $p$ is an odd prime.  As indicated in the paragraph
before Theorem~\ref{thm:main-thm}, we consider the pro-$p$ group
$G = F/N$, where 
\begin{itemize}
\item $F = \langle x,y \rangle$ is a free pro-$p$ group and
\item $N = [R,F] R^p \trianglelefteq_\mathrm{c} F$ for the kernel
  $R \trianglelefteq_\mathrm{c} F$ of the presentation
  $\pi \colon F \to W$ sending $x,y$ to the generators of the same
  name in~\eqref {equ:pro-p-wreath-prod-W}.
\end{itemize}
By producing generators for $R$ and $N$ as closed normal subgroups of
$F$ we obtain explicit presentations for the pro-$p$ groups $W$
and~$G$.

It is convenient to write $y_i = y^{x^i}$ for $i \in \mathbb{Z}$.
Setting
\begin{equation}\label{equ:R_k}
  R_k =\langle \{ x^{p^k}, y^p \} \cup \{ [y_0,y_i] \mid 1\le i \le
  \tfrac{p^k-1}{2} \} \rangle ^F \trianglelefteq_\mathrm{o} F
\end{equation}
for $k\in \mathbb{N}$, we obtain a descending chain of open normal
subgroups
\begin{equation}\label{equ:R}
  F \supseteq R_1\supseteq R_2 \supseteq \ldots
\end{equation}
with quotient groups
$F/R_k \cong W_k \cong C_p \mathrel{\wr} C_{p^k}$.  We put
\[
R = \langle \{ y^p \} \cup \{ [y_0,y_i] \mid i\in \mathbb{N} \}
\rangle^F \trianglelefteq_\mathrm{c} F,
\]
and observe that $R_k = \langle x^{p^k} \rangle^F \, R$ for each
$k\in \mathbb{N}$.  Since $x^{p^k} \to 1$ as $k \to \infty$, this
yields $R = \bigcap_{k \in \mathbb{N}} R_k$ and thus
$F/R \cong W \cong C_p \mathrel{\hat{\wr}} \mathbb{Z}_p$.
With hindsight there is no harm in taking $W_k = F/R_k$ for
$k \in \mathbb{N}$ and $W = F/R$.

Setting $N_k = [R_k,F]R_k^{\, p}$ for $k \in \mathbb{N}$, we observe
that
\begin{multline*}
  N_k = \langle \{ x^{p^{k+1}} , y^{p^2}, [x^{p^k},y], [y^p,x] \} \cup
  \{ [y_0,y_i]^p \mid 1\le i \le \tfrac{p^k-1}{2} \}  \\
  \cup \{ [y_0,y_i,x] \mid 1\le i \le \tfrac{p^k-1}{2}\} \cup \{
  [y_0,y_i,y] \mid 1\le i \le \tfrac{p^k-1}{2}\} \rangle ^F
  \trianglelefteq_\mathrm{o} F,
\end{multline*}
and as in \eqref{equ:R} we obtain a descending chain
$F \supseteq N_1 \supseteq N_2 \supseteq \ldots$ of open normal
subgroups.  Moreover, it follows that
$\bigcap_{k \in \mathbb{N}} N_k \supseteq [R,F]R^p=N$.  On the other
hand, if $z \not \in N$ then there exists an open normal subgroup
$K \trianglelefteq_\mathrm{o} F$ and $k \in \mathbb{N}$ such that
$z \not \in NK = [R_k,F]R_k^{\, p} K$, hence $z \not \in N_k$.
Thus we conclude that
\begin{equation*}
  \bigcap\nolimits_{k \in \mathbb{N}} N_k = [R,F]R^p = N.
\end{equation*}

Consequently, $G = F/N \cong \varprojlim G_k$, where
\begin{equation}\label{equ:pres-for-G-k}
  \begin{split}
    G_k = F/N_k
    & \cong  \langle x,y \mid \uline{\phantom{[}x^{p^{k+1}}},\;
    y^{p^2},\; \uline{~[x^{p^k},y]},\; [y^p,x];\\
    & \qquad \quad [y_0,y_i]^p,\; [y_0,y_i,x],\; [y_0,y_i,y] \quad
    \text{for $1\le i\le \tfrac{p^k-1}{2}$} \rangle
  \end{split}
\end{equation}
for $k \in \mathbb{N}$, and 
\begin{equation}\label{equ:pres-for-G}
  G \cong \langle x,y \mid y^{p^2},\; [y^p,x];\; [y_0,y_i]^p,\;
  [y_0,y_i,x],\; [y_0,y_i,y] \quad \text{for $i \in \mathbb{N}$} \rangle
\end{equation}
is a presentation of $G$ as a pro-$p$ group.  Indeed,
\[
\widetilde{N} = \langle \{ y^{p^2}, [y^p,x] \} \cup \{ [y_0,y_i]^p ,
[y_0,y_i,x] , [y_0,y_i,y] \mid i \in \N \} \rangle^F
\trianglelefteq_\mathrm{c} F
\]
satisfies, for each $k\in \mathbb{N}$,
\[
N_k= \langle x^{p^{k+1}} , [x^{p^k},y] \rangle^F \, \widetilde{N},
\]
where $x^{p^{k+1}}, [x^{p^k},y] \to 1$ as $k \to \infty$.  This yields
$\widetilde{N} = \bigcap_{k \in \N} N_k = N$.
To facilitate later use,
we have underlined the two relations in \eqref{equ:pres-for-G-k} that
do not yet occur in~\eqref{equ:pres-for-G}.

To summarise and supplement some of the notation introduced above,
we define
\begin{align*}
  Y & = \langle y_i \mid i \in \mathbb{Z} \rangle R
      \trianglelefteq_\mathrm{c} F,%
  & H & = Y/N \trianglelefteq_\mathrm{c} G, %
  & Z & = R/N \trianglelefteq_\mathrm{c} G. \\
  \intertext{Similarly for $k \in \mathbb{N}$ we set}
  Y_k & = \langle y_i \mid i \in \mathbb{Z} \rangle R_k
        \trianglelefteq_\mathrm{o} F, %
  & H_k & = Y_k/N_k \trianglelefteq G_k, %
  & Z_k &= R_k/N_k \trianglelefteq G_k.
\end{align*}
Diagrammatically, we have:
\[
\left. 
\begin{array}{ccc}
  F & \twoheadrightarrow & G \\ 
  \vert && \vert \\
  Y & \twoheadrightarrow & H \\
  \vert && \vert \\
  R & \twoheadrightarrow & Z 
\end{array}
\right\} G/Z \cong W \qquad\qquad W_k \cong G_k/Z_k \left\{
\begin{array}{ccc}
  G_k   & \twoheadleftarrow & F  \\
  \vert && \vert \\
  H_k  & \twoheadleftarrow & Y_k \\
  \vert && \vert \\
  Z_k & \twoheadleftarrow & R_k
\end{array}
\right.
\]\vspace*{-.8cm}\[
\begin{array}{ccc}
  \phantom{R} && \phantom{Z} \\
  \vert && \vert \\
  N & \twoheadrightarrow & 1 
\end{array}
\qquad\qquad\qquad\qquad\qquad\qquad\qquad\qquad\;
\begin{array}{ccc}
  \phantom{Z_k} && \phantom{R_k} \\
  \vert && \vert \\
  1 & \twoheadleftarrow & N_k
\end{array}
\]

%

\begin{lemma} \label{lem:centre-of-G} The centre of $G$ is $Z(G) = Z$,
  and $Z_k \le Z(G_k)$ for $k \in \mathbb{N}$.
\end{lemma}

\begin{proof}
  By construction, $Z \le Z(G)$ and $Z_k \le Z(G_k)$ for
  $k \in \mathbb{N}$.  From \eqref{equ:W-explicit} we see that
  $G/Z \cong W$ has trivial centre.  Therefore $Z = Z(G)$.
\end{proof}

In fact, $Z_k \lneqq Z(G_k)$ for $k \in \mathbb{N}$; see
Lemma~\ref{lem:w-elements} below.


\section{General description of the normal Hausdorff spectrum of the
  pro-$p$ group $G$
  and its finite direct powers} \label{sec:general-description}

We continue to use the notation set up in
Section~\ref{sec:presentation} to study the pro-$p$ group~$G$ and its
finite direct powers.

\begin{proposition} \label{pro:normal-subgroups-in-G} Let
  $K \trianglelefteq_\mathrm{c} G$ be a closed
  normal subgroup such that $K \not \subseteq Z$.  Then either $K$ is
  open in $H$ or $K$ is open in~$G$; in particular, $K \cap H
  \le_\mathrm{o} H$.  Furthermore, $[K \cap H,G] \le_\mathrm{o} H$.
\end{proposition}

\begin{proof}
  Lemma~\ref{lem:normal-subgroups-in-W} shows:
  $KZ \cap H \le_\mathrm{o} H$; hence it suffices to prove
  $K \cap Z \le_\mathrm{o} Z$.  Choose
  $\hat{y}_1, \hat{y}_2, \ldots \in H$, converging to $1$ modulo $Z$,
  and $m \in \mathbb{N}$ such that (the images of)
  $\hat{y}_1, \hat{y}_2, \ldots$ (modulo $Z$) yield a basis for the
  elementary abelian pro-$p$ group $H/Z$ and
  $\hat{y}_{m+1}, \hat{y}_{m+2}, \ldots$ generate $KZ \cap H$
  modulo~$Z$.

  Recall that $Z$ is central in $G$ and of exponent~$p$.  Thus
  $K \cap Z$ contains $\hat{y}_i^{\,p}$ and $[\hat{y}_i,\hat{y}_j]$
  for all $i, j \in \mathbb{N}$ with $i>m$.  Hence the finite set
  \[
 \{ \hat{y}_i^{\, p} \mid 1 \le i \le m \} \cup \{
  [\hat{y}_i,\hat{y}_j] \mid 1 \le i \le j \le m \}
  \]
  generates the elementary abelian group $Z$ modulo $K \cap Z$, and
  $K \cap Z \le_\mathrm{o} Z$.

  Finally, Lemma~\ref{lem:normal-subgroups-in-W} implies that
  $[K \cap H,G] \not \subseteq Z$.  Hence
  $[K \cap H,G] \le_\mathrm{o} H$.
\end{proof}

From Proposition~\ref{pro:normal-subgroups-in-G},
Lemma~\ref{lem:centre-of-G} and Lemmata~\ref{lem:interval}
and~\ref{lem:hdim-elem-ab} we deduce the general shape of the normal
Hausdorff spectrum of~$G$.

\begin{corollary} \label{cor:general-normal-hspec} Let $\mathcal{S}$
  be an arbitrary filtration series of~$G$.  Then the normal Hausdorff spectrum
  of $G$ has the form
  \[
 \hspec_\trianglelefteq^\mathcal{S}(G) = [0,\xi] \cup \{ \eta \} \cup
 \{ 1\},
  \]
  where $\xi = \hdim^{\mathcal{S}}_G(Z)$ and
  $\eta = \hdim^{\mathcal{S}}_G(H)$.
\end{corollary}

More generally we obtain a description of the normal Hausdorff
spectrum of finite direct powers
$G^{(m)} = G \times \overset{m} \ldots \times G$ of~$G$, with respect
to suitable `product filtration series'.  For any filtration series
$\mathcal{S} \colon G=S_0\supseteq S_1 \ldots$ of $G$ we consider the
naturally induced product filtration series on $G^{(m)}$ given by
\[
\mathcal{S}^{(m)} \colon G^{(m)} = G \times \overset{m} \ldots \times
G \supseteq S_1\times \overset{m}{\ldots}\times S_1 \supseteq
S_2\times \overset{m}{\ldots} \times S_2\supseteq \ldots.
\]
For a standard filtration series
$\mathcal{S} \in \{\mathcal{P},\mathcal{L},\mathcal{F},\mathcal{D} \}$
on $G$ the product filtration series $\mathcal{S}^{(m)}$ is actually the
corresponding standard filtration series on~$G^{(m)}$.

\begin{corollary}\label{cor:powers-of-G}
  Let $m \in \mathbb{N}$, and let
  $K \trianglelefteq_\mathrm{c} G^{(m)}$.  For $1\le j\le m$, let
  $\pi_j \colon G^{(m)} \to G$ be the canonical projection onto the
  $j$th factor and set
  \[
  \overline{K}(j) =
  \begin{cases}
    Z & \text{if $K \pi_j \subseteq Z$,} \\
    G & \text{otherwise,}
  \end{cases}
  \quad \text{and} \quad
  \underline{K}(j) =
  \begin{cases}
    1 & \text{if $K \pi_j \subseteq Z$,} \\
    H & \text{otherwise.}
  \end{cases}
  \]
  Then $K \le \prod_{j = 1}^m \overline{K}(j)$ and $K$ contains an
  open normal subgroup of $\prod_{j = 1}^m \underline{K}(j)$.
\end{corollary}

\begin{proof}
  Observe that
  \[
  [K\pi_1,G] \times \ldots \times [K\pi_m,G] = [K,G^{(m)}] \le K \le
  K\pi_1 \times \ldots \times K\pi_m.
  \]
  Thus $K$ is contained in $\prod_{j=1}^m \overline{K}(j)$, and it
  suffices to show that $[K\pi_j \cap H,G] \le_\mathrm{o} H$ for each
  $j$ with $K\pi_j \not \subseteq Z$.  This follows by
  Proposition~\ref{pro:normal-subgroups-in-G}.
\end{proof}

\begin{corollary}\label{cor:direct-product-spectrum}
  Let $m \in \mathbb{N}$, and let $\mathcal{S}$ be a filtration series
  of $G$ such that $\hdim^{\mathcal{S}}_G(H)=1$.  Then the normal
  Hausdorff spectrum of $G^{(m)}$ has the form 
  \[
  \hspec^{\mathcal{S}^{(m)}}_{\trianglelefteq}(G^{(m)}) = [0,\xi] \cup
  \bigcup_{1 \le l \le m-1}\big[\nicefrac{l}{m},\nicefrac{l +
  (m-l)\xi}{m} \big] \cup\{1\},
  \]
  where $\xi = \hdim^{\mathcal{S}}_G(Z)$.
\end{corollary}

\begin{proof}
 First let $K \trianglelefteq_\mathrm{c} G^{(m)}$, and define
 $\overline{K}(j), \underline{K}(j)$ for $1 \le j \le m$ as in
 Corollary~\ref{cor:powers-of-G}.  From $\hdim^{\mathcal{S}}_G(H)=1$
 we deduce that
 \begin{multline*}
   \nicefrac{l}{m} = \hdim^{\mathcal{S}^{(m)}}_{G^{(m)}} \big(
   \prod\nolimits_{j=1}^m \underline{K}(j) \big) \le
   \hdim^{\mathcal{S}^{(m)}}_{G^{(m)}}(K) \\ \le
   \hdim^{\mathcal{S}^{(m)}}_{G^{(m)}} \big( \prod\nolimits_{j=1}^m
   \overline{K}(j) \big) = \nicefrac{l}{m} + \nicefrac{m-l}{m} \, \xi,
 \end{multline*}
 where
 $l = \# \{ j \mid 1 \le j \le m \text{ and } \overline{K}(j) = G \}$.

 Conversely, for every $l \in \{0,1,\ldots,m\}$ and
 $\beta \in \big[\nicefrac{l}{m},\nicefrac{l + (m-l)\xi}{m} \big]$
 there is a normal subgroup
 \[
 K_\beta = G \times \overset{l} \ldots \times G \times U \times
 \overset{m-l} \ldots \times U \trianglelefteq_\mathrm{c} G^{(m)},
 \]
 where $U \le_\mathrm{c} Z$ for $l < m$ has
 $\hdim^\mathcal{S}_G(U) = \nicefrac{m}{m-l} \, (\beta -
 \nicefrac{l}{m}) \in [0,\xi]$;
 compare Corollary~\ref{cor:general-normal-hspec}.  This yields
 $\beta = \hdim^{\mathcal{S}^{(m)}}_{G^{(m)}}(K_\beta) \in
 \hspec^{\mathcal{S}^{(m)}}_\trianglelefteq(G^{(m)}).$
\end{proof}

Corollary~\ref{cor:direct-product-spectrum} shows that, once
$\hdim^\mathcal{S}_G(H) = 1$, the general shape (e.g.\ the number of
connected components) of the normal Hausdorff spectrum
$\hspec^{\mathcal{S}^{(m)}}_\trianglelefteq(G^{(m)})$ depends only on
the parameters $\xi = \hdim^{\mathcal{S}}_G(Z)$ and
$m \in \mathbb{N}$.  For instance, if $\xi < \nicefrac{1}{m}$, then
$\hspec^{\mathcal{S}^{(m)}}_\trianglelefteq(G^{(m)})$ is the union of
$m+1$ disjoint intervals, whereas for $\xi \ge \nicefrac{1}{2}$ we
obtain
$\hspec^{\mathcal{S}^{(m)}}_\trianglelefteq(G^{(m)}) =
[0,1-\nicefrac{(1-\xi)}{m}] \cup \{1\}$.

The proof of Theorem~\ref{thm:main-thm} in
Sections~\ref{sec:p-power-series} and \ref{sec:other-series} will give
$\hdim^\mathcal{S}_G(H) = 1$ for the standard filtrations
$\mathcal{S} \in \{ \mathcal{P}, \mathcal{D}, \mathcal{F} \}$ and
$\xi = \hdim^\mathcal{P}_G(Z) = \hdim^\mathcal{D}_G(Z)
=\nicefrac{1}{3}$
respectively $\xi = \hdim^\mathcal{F}_G(Z) = \nicefrac{1}{p+1}$; the
assertion for $H$ is already a consequence of
\cite[Prop.~4.2]{KlThZu19}.  We formulate a taylor-made corollary for
these situations.

\begin{corollary} \label{cor:mult-intervals}
  Let $m,n \in \mathbb{N}$ with $m\ge \max\{2, n-1\}$ and $n \ge 2$.
  Let $\mathcal{S}$ be a filtration series of $G$ such that
  $\hdim^{\mathcal{S}}_G(H)=1$ and
  $\hdim^{\mathcal{S}}_G(Z) = \nicefrac{1}{n}$.  Then
  \[
  \hspec^{\mathcal{S}^{(m)}}_{\trianglelefteq}( G^{(m)}) = \big[
  0,\tfrac{mn-(n-1)^2}{mn} \big] \cup \bigcup\limits_{m-n+2 \le l \le m-1}
  \big[ \tfrac{l}{m}, \tfrac{m +l(n-1)}{mn} \big] \cup \{1\}
  \]
  consists of $n$ disjoint intervals.
\end{corollary}

\begin{proof}
  From Corollary~\ref{cor:direct-product-spectrum}, we have
  \[
  \hspec^{\mathcal{S}^{(m)}}_{\trianglelefteq}( G^{(m)}) %
  = \big[ 0,\nicefrac{1}{n} \big] \cup \bigcup_{1 \le l \le m-1} \big[
  \nicefrac{l}{m}, \nicefrac{m+l(n-1)}{mn} \big] \cup \{1\}.
  \]

  For $m-n+1 \le l \le m-1$ it is easy to verify that
  \[
  \frac{m+l(n-1)}{mn} < \frac{l+1}{m}.
  \]
  Hence it suffices to show that
  \[
  \big[ 0,\nicefrac{1}{n} \big] \cup \bigcup_{1 \le l \le m-n+1} \big[
  \nicefrac{l}{m}, \nicefrac{m+l(n-1)}{mn} \big] = \big[
  0,\tfrac{mn-(n-1)^2}{mn} \big].
 \]
 For $m = n-1$ this reduces to
 $\big[ 0,\nicefrac{1}{n} \big] = \big[ 0,\tfrac{mn-(n-1)^2}{mn}
 \big]$.  Now suppose that $m \ge n$.  Then the claim follows from
 \[
 \nicefrac{1}{m} \le \nicefrac{1}{n} \qquad \text{and} \qquad
\nicefrac{l+1}{m} \le \nicefrac{m+l(n-1)}{mn} \quad \text{for
   $1 \le l \le m-n$.} \qedhere
 \] 
\end{proof}


\section{The normal Hausdorff spectrum of $G$ with respect to\\ the
  $p$-power series} \label{sec:p-power-series}

We continue to use the notation set up in
Section~\ref{sec:presentation} and establish that
$\xi = \hdim^{\mathcal{P}}_G(Z) = \nicefrac{1}{3}$ and
$\eta = \hdim^{\mathcal{P}}_G(H) = 1$, with respect to the $p$-power
series~$\mathcal{P}$.  In view of
Corollary~\ref{cor:general-normal-hspec} this proves
Theorem~\ref{thm:main-thm} for the $p$-power series.  Indeed,
$\hdim^{\mathcal{P}}_G(H) = 1$ is already a consequence of
\cite[Prop.~4.2]{KlThZu19}.  It remains to show that
\begin{equation} \label{equ:xi-1/3-P} \hdim^{\mathcal{P}}_G(Z) =
  \varliminf_{i\to \infty} \frac{\log_p \lvert ZG^{p^i} : G^{p^i}
    \rvert}{\log_p \lvert G : G^{p^i} \rvert} = \nicefrac{1}{3}.
\end{equation}

It is convenient to work with the finite quotients $G_k$,
$k \in \mathbb{N}$, introduced in Section~\ref{sec:presentation}.  Let
$k \in \mathbb{N}$.  From \eqref{equ:pres-for-G-k} and
\eqref{equ:pres-for-G} we observe that
\[ 
\lvert G : G^{p^k} \rvert = \lvert G_k : G_k^{\, p^k} \rvert.
\]
Heuristically, $G_k^{\, p^k}$ is almost trivial (see
Proposition~\ref{pro:index-Gk-Gk-hoch-p}) and the elementary abelian
$p$-group $Z_k$ requires roughly half the number of generators
compared to the elementary abelian $p$-group $H_k/Z_k$.  This suggests
that~\eqref{equ:pres-for-G} should be true. We now work out the
details.

First we compute the order of~$G_k$, using the notation from
Section~\ref{sec:presentation}.

\begin{lemma}\label{lem:order-Gk}
  The logarithmic order of $G_k$ is
  \[
  \log_p \lvert G_k \rvert = \tfrac{1}{2}(3p^k+2k+3).
  \]
  In particular,  
 \begin{equation*}
   Z_k = R_k/N_k = \langle \{ x^{p^k}, y^p \} \cup \{ [y_0,y_i] \mid 1 \le i \le
    \tfrac{p^k-1}{2} \} \rangle N_k / N_k \cong C_p \times
   \overset{\frac{p^k+3}{2}}{\ldots} \times C_p. 
 \end{equation*}
\end{lemma}

\begin{proof}
  Observe from $F/R_k \cong W_k \cong C_p \mathrel{\wr} C_{p^k}$ that
  \[
  \log_p \lvert G_k \rvert =\log_p \lvert F : R_k \rvert + \log_p
  \lvert R_k : N_k \rvert = k+ p^k + \log_p \lvert R_k : N_k \rvert.
  \]
  By construction, $R_k/N_k$ is elementary abelian of exponent~$p$.
  Moreover, \eqref{equ:R_k} shows that
  $\{ x^{p^k}, y^p \} \cup \{ [y_0,y_i] \mid 1 \le i \le
  (p^k-1)/2 \}$
  generates $R_k$ modulo $N_k$.  In order to prove that the generators
  are independent, we construct a factor group $\widetilde{G}_k$ of
  $G_k$ that has the maximal possible logarithmic order
  $\log_p \lvert \widetilde{G}_k \rvert = p^k+k+2+(p^k-1)/2$.

  Consider the finite $p$-group
  \[
  M = \langle \widetilde{y}_0, \widetilde{y}_1, \ldots ,
  \widetilde{y}_{p^k-1} \rangle = E/[\Phi(E),E]\Phi(E)^p,
  \]
  where $E$ is a free pro-$p$ group on $p^k$ generators with Frattini
  subgroup $\Phi(E) = [E,E] E^p$.  Then the images of
  $\widetilde{y}_0, \ldots , \widetilde{y}_{p^k-1}$ generate
  independently the elementary abelian quotient $M/\Phi(M)$ and the
  commutators $[\widetilde{y}_i, \widetilde{y}_j]$, for
  $0 \le i < j \le p^k-1$, together with the $p$th powers
  $\widetilde{y}_0^{\; p}, \ldots, \widetilde{y}_{p^k-1}^{\; p}$
  generate independently the elementary abelian group~$\Phi(M)$.  The
  latter can be checked by considering homomorphisms from $M$ onto
  groups of the form $C_p^{\, p^k-1} \times C_{p^2}$ and
  $C_p^{\, p^k -2} \times \mathrm{Heis}(\mathbb{F}_p)$, where
  $\mathrm{Heis}(\mathbb{F}_p)$ denotes the group of upper
  unitriangular $3 \times 3$ matrices over $\mathbb{F}_p$.  Next
  consider the action of the cyclic group
  $X = \langle \widetilde{x} \rangle \cong C_{p^{k+1}}$, with kernel
  $\langle \widetilde{x}^{p^k} \rangle \cong C_p$, on $M$ that is
  induced by
  \begin{align*}
    \widetilde{y}_i^{\; \widetilde{x}} =
    \begin{cases}
      \widetilde{y}_{i+1} & \text{if $0 \le i \le p^k-2$,} \\
      \widetilde{y}_0 & \text{if $i = p^k-1$.}
    \end{cases}
  \end{align*}
  This induces a permutation action on our chosen basis for the
  elementary abelian group $\Phi(M)$; the orbits are given by
  \[
  [\widetilde{y}_i,\widetilde{y}_j] \equiv_X [\widetilde{y}_{i'},\widetilde{y}_{j'}]
  \;\;\longleftrightarrow\;\; j-i \equiv_{p^k} j'-i' \qquad \text{and}
  \qquad \widetilde{y}_0^{\; p} \equiv_X \ldots \equiv_X
  \widetilde{y}_{p^k-1}^{\; p}.
  \]
  We define $\widetilde{M} = M/[\Phi(M),X]$ and, for simplicity,
  continue to write $\widetilde{y}_0, \ldots, \widetilde{y}_{p^k-1}$
  for the images of these elements in~$\widetilde{M}$.  Then
  \begin{enumerate}
  \item[$\circ$] the images of
    $\widetilde{y}_0, \ldots , \widetilde{y}_{p^k-1}$ generate
    independently the elementary abelian quotient
    $\widetilde{M}/\Phi(\widetilde{M})$ and
  \item[$\circ$] the elements $[\widetilde{y}_0, \widetilde{y}_i]$,
    for $1 \le i  \le  (p^k-1)/2$, together with $\widetilde{y}_0^{\; p}$
    generate independently the elementary abelian
    group~$\Phi(\widetilde{M})$.
  \end{enumerate}
  In particular, this yields
  $\log_p \lvert \widetilde{M} \rvert = p^k + (p^k-1)/2 + 1$.

  Finally, we put $\widetilde{y} = \widetilde{y}_0$ and form the
  semidirect product
  \[
  \widetilde{G}_k = \langle \widetilde{x}, \widetilde{y} \rangle = X
  \ltimes \widetilde{M}
  \]
  with the induced action.  Upon replacing $x,y$ by
  $\widetilde{x},\widetilde{y}$, we see that all the defining
  relations of $G_k$ in \eqref{equ:pres-for-G-k} are valid in
  $\widetilde{G}_k$.  Since
  $\log_p \lvert G_k \rvert \le p^k + k + 2 + (p^k-1)/2 = \log_p
  \lvert \widetilde{G}_k \rvert$,
  we conclude that $G_k \cong \widetilde{G}_k$.
\end{proof}

Our next aim is to prove the following structural result.

\begin{proposition} \label{pro:index-Gk-Gk-hoch-p} In the set-up from
  Section~\ref{sec:presentation}, for $k\ge 2$, the subgroup
  $G_k^{\, p^k} \le G_k$ is elementary abelian and central in $G_k$;
  it is generated independently by $x^{p^k}$,
  $w = y_{p^k-1} \cdots y_1y_0$ and
  $v = w \cdot y_{p^k-1}^{-1} \cdots y_1^{-1} y_0^{-1}$.

  Consequently
  \[
  G_k^{\, p^k} \cong C_p \times C_p \times C_p, \qquad \log_p \vert
  G_k : G_k^{\, p^k} \rvert = \log_p \lvert G_k \rvert - 3
  \]
  and
  \begin{equation*}
    \begin{split}
      G_k / G_k^{\, p^k} & \cong \langle x,y \mid x^{p^{k}},\;
      y^{p^2},\; 
      [y^p,x],\; w(x,y),\; v(x,y);\\
      & \qquad \quad [y_0,y_i]^p,\; [y_0,y_i,x],\; [y_0,y_i,y] \quad
      \mathrm{for } 1\le i\le \tfrac{p^k-1}{2} \rangle.
    \end{split}
  \end{equation*}
\end{proposition}



The proof requires a series of lemmata.

\begin{lemma} \label{lem:w-elements} The elements
  \[
    w = y_{p^k-1}\cdots y_1y_0 \qquad \text{and} \qquad
    w' = y_{p^k-1}^{\, -1} \cdots y_1^{\, -1} y_0^{\, -1}
  \]
  are of order $p$ in $G_k$ and lie in $G_k^{\, p^k} \cap Z(G_k)$.
\end{lemma}

\begin{proof}
  Recall that $H_k = \langle y_0, y_1, \ldots, y_{p^k-1} \rangle Z_k \le G_k$
  and observe that $[H_k,H_k]$ is a central subgroup of exponent $p$
  in~$G_k$.  Furthermore, $[y^p,x] =1$ implies
  $y_{p^k-1}^{\, p} = \ldots = y_0^{\, p}$ in~$G_k$.  Thus
  \eqref{equ:commutator-formula-1} yields
  \[
  w^p = y_{p^k-1}^{\, p} \cdots y_1^{\, p} y_0^{\, p} = y^{p^{k+1}} = 1.
  \]
  As $w \ne 1$ we deduce that $w$ has order~$p$.  Likewise one
  shows that $w'$ has order~$p$.

  Clearly, $w = x^{-p^k} (xy)^{p^k}$ and
  $w' = x^{-p^k} (x y^{-1})^{p^k}$ lie in $G_k^{\, p^k}$.  In order to
  prove that $w$ is central, it suffices to check that $w$ commutes
  with the generators $x$ and $y$ of $G_k$.  First we observe that, for
  $1 \le i \le p^k-1$, the relation $[y_0,y_i,x]=1$ implies
 \begin{equation}\label{equ:note}
   [y_0,y_{p^k-i}]^{-1} = [y_{p^k-i},y_0] = [y_0,y_i]^{x^{-i}} =
   [y_0,y_i] \qquad \text{in $G_k$.} 
 \end{equation}
 Since $[H_k,H_k]$ is central in $G_k$, we deduce inductively that
 \begin{align*}
   [w,x] & = (y_{p^k-1} \cdots
           y_1y_0)^{-1} (y_{p^k-1} \cdots y_1 y_0)^x \\
         &  = y_0^{\, -1} y_1^{\, -1} \cdots y_{p^k-2}^{\, -1} \cdot
           y_{p^k-1}^{\, -1} y_0 y_{p^k-1} \cdot y_{p^k-2} \cdots
           y_2 y_1 \\
         & = y_0^{\, -1} y_1^{\, -1} \cdots y_{p^k-2}^{-1} \cdot y_0 [y_0,
           y_{p^k-1}] \cdot y_{p^k-2} \cdots y_2y_1 \\
         & = y_0^{-1} y_1^{-1} \cdots y_{p^k-3}^{-1} \cdot
           y_{p^k-2}^{\, -1} y_0 y_{p^k-2} \cdot
           y_{p^k-3} \cdots y_2 y_1 \cdot [y_0,y_{p^k-1}]\\
         & \quad \vdots \\
         & = [y_0,y_1] [y_0,y_2] \cdots [y_0,y_{p^k-2}] [y_0,y_{p^k-1}]\\
         & =1 && \text{by \eqref{equ:note}.} 
  \end{align*}
  Likewise, using the relation $[y_0,y_i,y]=1$ and \eqref{equ:note}, we
  obtain
  \[
  [w,y] = [y_{p^k-1} \cdots y_1y_0, y_0] = [y_{p^k-1},y_0]
  [y_{p^k-2},y_0] \cdots [y_1,y_0] =1.
  \]
  A similar computation can be carried out for~$w'$.
\end{proof}

\begin{lemma} \label{lem:contains} Putting
  \[
  v = ww' = y_{p^k-1}\ldots y_1y_0\cdot y_{p^k-1}^{\,-1} \ldots y_1^{\, -1}
  y_0^{\, -1},
  \]
  the subgroup $\langle x^{p^k}, w, v \rangle \le G_k$ is isomorphic
  to $C_p \times C_p \times C_p$ and lies in $G_k^{\, p^k} \cap Z(G_k)$.
\end{lemma}

\begin{proof}
  From the presentation~\eqref{equ:pres-for-G-k} and from
  Lemma~\ref{lem:w-elements} it is clear that the subgroup
  $\langle x^{p^k}, w, v \rangle \le G_k$ is elementary abelian and
  lies in $G_k^{\, p^k} \cap Z(G_k)$.  Furthermore, in order to prove
  that
  $\langle x^{p^k}, w, v \rangle \cong C_p \times C_p \times C_p$, it
  suffices to establish that $v\ne 1$.

  Upon a similar rearrangement and cancellation as in the proof of
  Lemma~\ref{lem:w-elements}, we obtain
  \begin{equation*}
    v = \prod_{i=0}^{p^k-2} [y_i,y_{p^k-1}^{-1}] [y_i,y_{p^k-2}^{-1}] \cdots
    [y_i,y_{i+1}^{-1}].
  \end{equation*}
  Recall that all commutators appearing in the above product are
  central in~$G_k$.  In particular, we have
  $[y_0,y_{p^k-j}] = [y_0,y_{p^k-j}]^{x^i}=[y_i,y_{p^k-j+i}]$, for
  $1\le j\le p^k-1$ and $1\le i\le j-1$.  This gives
  \begin{align*}
    v & = [y_0,y_{p^k-1}^{-1}] \, [y_0,y_{p^k-2}^{-1}]^2 \cdots 
        [y_0,y_1^{-1}]^{p^k-1} \\
      & = [y_0,y_{p^k-1}]^{-1} \, [y_0,y_{p^k-2}]^{-2} \,
        \cdots [y_0,y_1]^{1-p^k} \\
      & = [y_0,y_1] \, [y_0,y_2]^{2} \, \cdots
        [y_0,y_{(p^k-1)/2}]^{\frac{p^k-1}{2}} \\
      & \qquad \cdot
        [y_0,y_{(p^k-1)/2}]^{(p^k-1)/2} \cdots
        [y_0,y_2]^{2} \, [y_0,y_1] && \text{by \eqref{equ:note}}\\
      & =[y_0,y_1]^2 \, [y_0,y_2]^4 \cdots [y_0,y_{(p^k-1)/2}]^{p^k-1}.
  \end{align*}
  Taking note of the second statement in Lemma~\ref{lem:order-Gk}, it
  follows that $v\ne 1$.
\end{proof}

\begin{lemma} \label{lem:gamma-2-exp-p} The group
  $\gamma_2(G_k) \le G_k$ has exponent~$p$.
\end{lemma}

\begin{proof}
  Recall that
  $H_k = \langle y_0, y_1, \ldots, y_{p^k-1} \rangle Z_k \le G_k$
  satisfies: $[H_k,H_k]$ is a central subgroup of exponent $p$
  in~$G_k$.  Since $p$ is odd,
  \eqref{equ:commutator-formula-1} shows that it
  suffices to prove that $[y,x]$ has order~$p$.  But
  $[y,x] = y_0^{\, -1} y_1$; thus \eqref{equ:commutator-formula-1} and
  $y_0^{\, p} = x^{-1} y_0^{\, p} x = y_1^{\, p}$ imply
  $[y,x]^p = y_0^{\, -p} y_1^{\, p} = 1$.
\end{proof}

\begin{lemma} \label{lem:gamma-m-modulo-gamma-m+1}
  The group $G_k$ has nilpotency class~$p^k$, and
  $\gamma_m(G_k) / \gamma_{m+1}(G_k)$ is elementary abelian of rank at
  most~$2$ for $2 \le m \le p^k$.
\end{lemma}

\begin{proof}
  Let $2 \le m \le p^k$.  Since $G_k$ is a central extension of $Z_k$
  by $W_k$, we deduce from Proposition~\ref{pro:series-Wk} that
  \[
  \gamma_m(G_k) = \langle [y,x,\overset{m-1}{\ldots},x],
  [y,x,\overset{m-2}{\ldots},x,y] \rangle \; \gamma_{m+1}(G_k),
  \]
  and Lemma~\ref{lem:gamma-2-exp-p} shows that
  $\gamma_m(G_k) / \gamma_{m+1}(G_k)$ is elementary abelian of rank at
  most~$2$.  Again by Proposition~\ref{pro:series-Wk}, the nilpotency
  class of $G_k$ is at least~$p^k$.  Moreover,
  $\gamma_{p^k}(G_k) Z_k = \langle w \rangle Z_k$, where
  $w \in Z(G_k)$ by Lemma~\ref{lem:w-elements}.  We conclude
  that $G_k$ has nilpotency class precisely~$p^k$.
\end{proof}
 
\begin{lemma} \label{lem:contained} The group $G_k$ satisfies
  \begin{equation*}
    G_k^{\, p} \subseteq \langle x^p, y^p  \rangle \gamma_p(G_k) \qquad
    \text{and} \qquad   G_k^{\, p^j} \subseteq \langle x^{p^j} \rangle
    \gamma_{p^j}(G_k) \quad \text{for $j \ge 2$.}
  \end{equation*}
\end{lemma}

\begin{proof}
  Recall that
  $H_k = \langle y_0, y_1, \ldots, y_{p^k-1} \rangle Z_k \le G_k$ has
  exponent~$p^2$, and observe that
  Proposition~\ref{pro:standard-commutator-id} together with
  Lemma~\ref{lem:gamma-2-exp-p} yields
  $H_k^{\, p} = \langle y^p \rangle$.  Every element $g \in G$ is of
  the form $g = x^m h$, with $0 \le m < p^{k+1}$ and $h \in H_k$.
  Using \eqref{equ:commutator-formula-3}, based on
    Proposition~\ref{pro:standard-commutator-id} and
    Lemma~\ref{lem:gamma-2-exp-p}, 
  we conclude that
  \[
  g^p = (x^m h)^p \equiv x^{m p} h^p \in \langle x^p, y^p \rangle \mod
  \gamma_p(G_k),
  \]
  and for $j \ge 2$,
  \[
  g^{p^j} = (x^m h)^{p^j} \equiv x^{m p^j} h^{p^j} = x^{mp^j} \in
  \langle x^{p^j} \rangle \mod \gamma_{p^j}(G_k). \qedhere
  \]
\end{proof}

\begin{proof}[Proof of Proposition~\ref{pro:index-Gk-Gk-hoch-p}]
  Apply
  Lemmata~\ref{lem:contains}, \ref{lem:gamma-m-modulo-gamma-m+1}
  and~\ref{lem:contained}.
\end{proof}

From Lemma~\ref{lem:order-Gk} and
Proposition~\ref{pro:index-Gk-Gk-hoch-p} we deduce that
\[
\log_p \lvert G : G^{p^k} \rvert = \log_p \lvert G_k : G_k^{\, p^k}
\rvert = \tfrac{1}{2} (3p^k+2k-3). 
\]
On the other hand, we observe from Proposition~\ref{pro:series-Wk}
that
\[
\log_p \lvert G : Z G^{p^k} \rvert = \log_p \lvert W_k : W_k^{\, p^k}
\rvert =  p^k+k-1,
\]
hence
\[
\log_p \lvert ZG^{p^k} : G^{p^k} \rvert = \tfrac{1}{2} (3p^k+2k-3) -
(p^k+k-1) = \tfrac{1}{2}(p^k-1).
\]

Thus \eqref{equ:xi-1/3-P} follows from
\begin{equation} \label{equ:xi-for-P} \varliminf_{i\to \infty}
  \frac{\log_p \lvert ZG^{p^i} : G^{p^i} \rvert}{\log_p \lvert G :
    G^{p^i} \rvert} = \lim_{i \to \infty}
  \frac{\tfrac{1}{2}(p^i-1)}{\tfrac{1}{2} (3p^i+2i-3)} =
  \nicefrac{1}{3}.
\end{equation}

\begin{remark}
  In the literature, one sometimes encounters a variant of the
  $p$-power series, the \emph{iterated} $p$-power series of $G$ which
  is recursively given by
  \[
  \mathcal{I} \colon I_0(G) = G, \quad \text{and} \quad I_j(G) =
  I_{j-1}(G)^p \quad \text{for $j \geq 1$.}
  \]
  By a small modification of the proof of Lemma~\ref{lem:contained} we
  obtain inductively 
    \begin{equation*}
    I_j(G_k) \subseteq \big(\langle x^{p^{j-1}} \rangle
       \gamma_{p^{j-1}}(G_k) \big)^p \subseteq  \langle x^{p^j} \rangle
       \gamma_{p^j}(G_k) \quad \text{for $j \ge 2$,}
  \end{equation*}
  based on the commutator identities~\eqref{equ:commutator-formula-3}
  for~$r=1$.  With Proposition~\ref{pro:index-Gk-Gk-hoch-p} and
  Lemma~\ref{lem:gamma-m-modulo-gamma-m+1} this yields
  $G_k^{\, p^k} \subseteq I_k(G_k) \subseteq \langle x^{p^k} \rangle
  \gamma_{p^k}(G_k) = G_k^{\, p^k}$.
  We conclude that the $p$-power series $\mathcal{P}$ and the iterated
  $p$-power series $\mathcal{I}$ of~$G$ coincide.

  One may further note another natural filtration series
  $\mathcal{N} \colon N_i$, $i\in \mathbb{N}_0$, of $G$, consisting of
  the open normal subgroups defined in Section~\ref{sec:presentation},
  where we set $N_0=G$.  As $N_i \le G^{p^i}$ with
  $\log_p \lvert G^{p^i} : N_i \rvert \le 4$ for all
  $i \in \mathbb{N}_0$, we see that the filtration series
  $\mathcal{P}$ and $\mathcal{N}$ induce the same Hausdorff dimension
  function on~$G$.
\end{remark}


\section{The normal Hausdorff spectra of $G$ with respect to the
  lower $p$-series, the dimension subgroup series and the Frattini series} \label{sec:other-series}

We continue to use the notation set up in
Section~\ref{sec:presentation} and work with the finite quotients
$G_k$, $k \in \mathbb{N}$, of the pro-$p$ group~$G$.  Our aim is to
pin down the lower central series, the lower $p$-series, the dimension
subgroup series and the Frattini series of~$G_k$.  Subsequently, it
will be easy to complete the proof of Theorem~\ref{thm:main-thm}.

\begin{proposition} \label{pro:lower-central-Gk} The group $G_k$ is
  nilpotent of class $p^k$; its lower central series satisfies
  \[
  G_k = \gamma_1(G_k) = \langle x,y \rangle \; \gamma_2(G_k) \quad
  \text{with} \quad G_k/\gamma_2(G_k)  \cong C_{p^{k+1}} \times
  C_{p^2}
 \]
 and, for $1 \le i \le (p^k-1)/2$,
 \begin{align*}
   \gamma_{2i}(G_k) & = \langle [y,x,\overset{2i-1}{\ldots},x] \rangle
                      \; \gamma_{2i+1}(G_k), \\
   \gamma_{2i+1}(G_k) & = \langle [y,x,\overset{2i}{\ldots},x],
                        [y,x,\overset{2i-1}{\ldots},x,y] \rangle \;
                        \gamma_{2i+2}(G_k)
 \end{align*} 
 with
 \[
 \gamma_{2i}(G_k)/\gamma_{2i+1}(G_k) \cong C_p \quad \text{and} \quad
 \gamma_{2i+1}(G_k)/\gamma_{2i+2}(G_k) \cong C_p \times C_p.
 \]
\end{proposition}

\begin{proof}
  By Lemma~\ref{lem:gamma-m-modulo-gamma-m+1} the nilpotency class of
  $G_k$ is~$p^k$.  From $G_k = \langle x,y \rangle$ it is clear that
  $\gamma_2(G_k) = \langle [x,y] \rangle \gamma_3(G_k)$, and
  \eqref{equ:pres-for-G-k} gives
  $G_k/\gamma_2(G_k) \cong C_{p^{k+1}} \times C_{p^2}$.  From
  Lemma~\ref{lem:order-Gk} we know that
  \[
   \log_p \lvert G_k \rvert = (3p^k+2k+3)/2 = ((k+1) + 2) +
   \tfrac{p^k-1}{2} (1 + 2),
  \]
  and the proof of Lemma~\ref{lem:gamma-m-modulo-gamma-m+1} shows that
  \[
  \gamma_m(G_k) = \langle [y,x,\overset{m-1}{\ldots},x],
  [y,x,\overset{m-2}{\ldots},x,y] \rangle \; \gamma_{m+1}(G_k) \qquad
  \text{for $2 \le m \le p^k$.}
  \]

 Consequently, it suffices to prove that
  $[y,x,\overset{m-2}{\ldots},x,y] \in \gamma_{m+1}(G_k)$ whenever $m$
  is even.  More generally, we consider the elements
  \[
  b_{j,m} = [ [y,x,\overset{m-2}{\ldots},x]^{x^j},y] \qquad \text{for
    $2 \le m \le p^k$ and $j \in \mathbb{N}_0$.}
  \]

  Writing $e_i = [y_0,y_i] \in Z_k \subseteq Z(G_k)$ for
  $i \in \mathbb{Z}$, we recall from
  Lemma~\ref{lem:order-Gk} that
  \[
  b_{j,m} \in [H_k,H_k] = \langle
  e_i \mid 1 \le i \le \tfrac{p^k-1}{2} \rangle \cong C_p \times
  \overset{\frac{p^k-1}{2}}{\ldots} \times C_p.
  \]
  Induction on $m$ shows that
  \[
  [y,x,\overset{m-2}{\ldots},x] \equiv  \prod_{i=0}^{m-2}
  y_i^{\, (-1)^{m+i} \binom{m-2}{i}} \quad
  \text{modulo~$Z_k \subseteq Z(G_k)$,}
  \]
  and we deduce that
  \begin{equation} \label{equ:bjm-formula} b_{j,m} = \Big[
    \prod\nolimits_{i=0}^{m-2} y_{j+i}^{\, (-1)^{m+i} \binom{m-2}{i}}
    , y \Big] = \prod\nolimits_{i=0}^{m-2} e_{j+i}^{\, (-1)^{m+i+1}
      \binom{m-2}{i}}.
  \end{equation}
  The identities
  \[
   \binom{m-2}{i} - 2 \binom{m-1}{i} + \binom{m}{i} = 
   \binom{m-2}{i-2}
  \]
  imply that
  \begin{equation} \label{equ:recursion}
  b_{j,m} \equiv b_{j,m} b_{j,m+1}^{\, 2} b_{j,m+2} = b_{j+2,m} \quad
  \text{modulo $\gamma_{m+1}(G_k)$.}
  \end{equation}

  Now suppose that $m$ is even, and recall that $p \ne 2$.
  From~\eqref{equ:recursion} we obtain inductively
  $[y,x,\overset{m-2}{\ldots},x,y] = b_{0,m} \equiv b_{j_0,m}$ modulo
  $\gamma_{m+1}(G_k)$ for
  \[
  j_0 =
  \begin{cases}
    \frac{p^k+1}{2} - \frac{m}{2} & \text{if $p^k+1-m \equiv_4
      0$,} \\
    \frac{p^k+3}{2} - \frac{m}{2} & \text{if $p^k+1-m \equiv_4
      2$.}
  \end{cases}
  \]

  Consequently, it suffices to prove that
  $b_{j_0,m} \in \gamma_{m+1}(G_k)$.  
    First suppose that $p^k+1 \equiv_4 m$ and hence
    $j_0 = \frac{p^k+1}{2} - \frac{m}{2}$. 
  From~\eqref{equ:bjm-formula} and
  \eqref{equ:note} we see that
  \begin{align*}
    b_{j_0,m} & = \prod\nolimits_{i=0}^{m/2-1} e_{j_0+i}^{\,
                (-1)^{i+1} \binom{m-2}{i}} \prod\nolimits_{i=
                m/2}^{m-2} e_{p^k - (j_0+i)}^{\, 
                (-1)^i \binom{m-2}{i}} \\
              & = \prod\nolimits_{i=0}^{m/2-1} e_{j_0+i}^{\,
                (-1)^{i+1} \binom{m-2}{i}} \prod\nolimits_{i=
                m/2}^{m-2} e_{j_0 + (m-1-i)}^{\, 
                (-1)^{m-i} \binom{m-2}{(m-1-i)-1}} \\
              & = \prod\nolimits_{i=0}^{m/2-1} e_{j_0+i}^{\,
                (-1)^{i+1} \binom{m-2}{i}} \prod\nolimits_{i'=
                1}^{m/2-1} e_{j_0 + i'}^{\, 
                (-1)^{i'+1} \binom{m-2}{i'-1}} \\
              & = \prod\nolimits_{i=0}^{m/2-1} e_{j_0+i}^{\,
                (-1)^{i+1} \binom{m-1}{i}} 
  \end{align*}
  and similarly
  \begin{align*}
    b_{j_0,m+1}^{\, -1} & = \prod\nolimits_{i=0}^{m/2-1} e_{j_0+i}^{\,
                (-1)^{i+1} \binom{m-1}{i}} \prod\nolimits_{i=
                m/2}^{m-1} e_{p^k - (j_0+i)}^{\, 
                (-1)^i \binom{m-1}{i}} \\
              & = \prod\nolimits_{i=0}^{m/2-1} e_{j_0+i}^{\,
                (-1)^{i+1} \binom{m-1}{i}} \prod\nolimits_{i=
                m/2}^{m-1} e_{j_0 + (m-1-i)}^{\, 
                (-1)^{m-i} \binom{m-1}{m-1-i}} \\
              & = \prod\nolimits_{i=0}^{m/2-1} e_{j_0+i}^{\,
                (-1)^{i+1} \binom{m-1}{i}} \prod\nolimits_{i'=
                0}^{m/2-1} e_{j_0 + i'}^{\, 
                (-1)^{i'+1} \binom{m-1}{i'}} \\
              & = \Big( \prod\nolimits_{i=0}^{m/2-1} e_{j_0+i}^{\,
                (-1)^{i+1} \binom{m-1}{i}} \Big)^2
  \end{align*}
  Hence $b_{j_0,m}^{\, 2} = b_{j_0,m+1}^{\, -1} \in \gamma_{m+1}(G_k)$,
  and $p \ne 2$ implies $b_{j_0,m} \in \gamma_{m+1}(G_k)$.

    In the remaining case $p^k+1 \equiv_4 m + 2$ we have
    $j_0 = \frac{p^k+3}{2} - \frac{m}{2}$, and a slight variation of
    the argument above shows that
    $b_{j_0,m}^{\, 2} = b_{j_0-1,m+1}$, hence
    $b_{j_0,m} \in \gamma_{m+1}(G_k)$.
\end{proof}

\begin{corollary} \label{cor:gamma-central-part}
  For $2 \le m \le p^k$ and
  $\nu(m) = \lfloor \frac{1}{2} (p^k-m+2) \rfloor$, we have
  \[
    \gamma_m(G_k) \cap Z_k = \langle [y,x,\overset{2j-1}{\ldots},x,y]
    \mid \lfloor \nicefrac{m}{2} \rfloor \le j \le
    \nicefrac{(p^k-1)}{2} \rangle \cong C_p^{\, \nu(m)}
  \]
  and
  $\gamma_m(G_k) \cap Z(G_k) = \langle [y,x,\overset{p^k-1}{\ldots},x]
  \rangle \times (\gamma_m(G_k) \cap Z_k) \cong C_p^{\, \nu(m)+1}$.
  In particular,
  \[
  [y,x,\overset{m-2}{\ldots},x,y] \in \langle [y,x,\overset{2j-1}{\ldots},x,y]
  \mid \nicefrac{m}{2} \le j \le \nicefrac{(p^k-1)}{2}
  \rangle \qquad \text{for $m \equiv_2 0$.}
  \]
\end{corollary}

\begin{proof}
  Clearly, all non-trivial elements of the form $[y,x,\ldots,x,y]$ are
  central and of order~$p$.  By Proposition~\ref{pro:lower-central-Gk}
  and Lemma~\ref{lem:gamma-2-exp-p}, also
  $[y,x,\overset{p^k-1}{\ldots},x]$ is central and of order~$p$.
  Moreover, Proposition~\ref{pro:lower-central-Gk} shows that every
  $g \in \gamma_2(G_k)$ can be written as
  \[
  g = \prod\nolimits_{i=1}^{p^k-1} [y,x,\overset{i}{\ldots},x]^{\alpha(i)}
  \prod\nolimits_{j=1}^{(p^k-1)/2}
  [y,x,\overset{2j-1}{\ldots},x,y]^{\beta(j)},
  \]
  where $\alpha(i), \beta(j) \in \{0,1,\ldots,p-1\}$ are uniquely
  determined by~$g$.  Furthermore, $g$ is central if and only if
  $\alpha(i) = 0$ for $1 \le i \le p^k-2$, and $g \in Z_k$ if and only
  if $\alpha(i) = 0$ for $1 \le i \le p^k-1$.
\end{proof}

\begin{corollary} \label{cor:lower-p-central-Gk} The lower $p$-series
  of $G_k$ has length $p^k$ and satisfies:
  \begin{align*}
    & G_k  = P_1(G_k)  = \langle x,y \rangle \; P_2(G_k) %
    &  \text{with} \quad %
    & G_k/P_2(G_k)  \cong C_p \times C_p, \\
    & P_2(G_k) = \langle x^p, y^p, [y,x] \rangle \; P_3(G_k) %
    & \text{with} \quad %
    & P_2(G_k)/P_3(G_k)  \cong C_p \times C_p \times C_p,
 \end{align*}
 and, for $3 \le i \le p^k$, the $i$th term is
 $P_i(G_k) = \langle x^{p^{i-1}} \rangle \gamma_i(G_k)$ so that
 \[
 P_i(G_k) =
 \begin{cases}
   \langle x^{p^{i-1}}, [y,x,\overset{i-1}{\ldots},x] \rangle
   \; P_{i+1}(G_k) & \text{if $i \equiv_2 0$ and $i \le k+1$,} \\
   \langle x^{p^{i-1}}, [y,x,\overset{i-1}{\ldots},x],
   [y,x,\overset{i-2}{\ldots},x,y] \rangle \; P_{i+1}(G_k) & \text{if
     $i \equiv_2 1$ and $i \le k+1$,} \\
   \langle [y,x,\overset{i-1}{\ldots},x] \rangle
   \; P_{i+1}(G_k) & \text{if $i \equiv_2 0$ and $i > k+1$,} \\
   \langle [y,x,\overset{i-1}{\ldots},x],
   [y,x,\overset{i-2}{\ldots},x,y] \rangle \; P_{i+1}(G_k) & \text{if
     $i \equiv_2 1$ and $i > k+1$}
 \end{cases}
 \]
 with
 \[
 P_i(G_k)/P_{i+1}(G_k) \cong 
 \begin{cases}
   C_p \times C_p & \text{if $i \equiv_2 0$ and $i \le k+1$,} \\
   C_p \times C_p \times C_p & \text{if $i \equiv_2 1$ and
     $i \le k+1$,} \\
   C_p & \text{if $i \equiv_2 0$ and $i > k+1$,} \\
   C_p \times C_p & \text{if $i \equiv_2 1$ and $i > k+1$}.
 \end{cases}
 \]
\end{corollary}

\begin{proof}
  The descriptions of $G_k / P_2(G_k)$ and $P_2(G_k) / P_3(G_k)$ are
  straightforward.  Let $i \ge 3$.  Clearly,
  $P_i(G_k) \supseteq \langle x^{p^{i-1}} \rangle \gamma_i(G_k)$.  In
  view of Proposition~\ref{pro:lower-central-Gk}, it suffices to prove
  that $x^{p^{i-1}}$ is central modulo~$\gamma_{i+1}(G_k)$.  Indeed,
  from Lemma~\ref{lem:gamma-2-exp-p} and
  Proposition~\ref{pro:standard-commutator-id} (recall that $p>2$) we
  obtain
  \[
  [x^{p^{i-1}},y] \equiv [x,y]^{p^{i-1}} 
    = 1 
  \quad \text{modulo } \gamma_{p^{i-1}+1}(G_k) \subseteq
  \gamma_{i+1}(G_k). \qedhere
  \]
\end{proof}

\begin{corollary}\label{cor:dim-subgroup}
  The dimension subgroup series of $G_k$ has length~$p^k$.  For
  $1\le i \le p^k$, the $i$th term is
  $D_i(G_k) = G_k^{\, p^{l(i)}} \gamma_i(G_k)$, where
  $l(i) = \lceil\log_p i \rceil$.

  Furthermore, if $i$ is not a power of~$p$, equivalently if
  $l(i+1) = l(i)$, then
  $D_i(G_k)/D_{i+1}(G_k) \cong \gamma_i(G_k) / \gamma_{i+1}(G_k)$ so that
  \[
  D_i(G_k) =
  \begin{cases}
    \langle [y,x,\overset{i-1}{\ldots},x] \rangle  D_{i+1}(G_k) &
    \text{if $i\equiv_2 0$,} \\
    \langle [y,x,\overset{i-1}{\ldots},x],
    [y,x,\overset{i-2}{\ldots},x,y] \rangle D_{i+1}(G_k) & \text{if $i\equiv_2 1$,} 
  \end{cases}
  \]
  with
  \[
  D_i(G_k)/D_{i+1}(G_k) \cong
  \begin{cases}
    C_p & \text{if $i\equiv_2 0$,} \\
    C_p \times C_p & \text{if $i \equiv_2 1$} 
 \end{cases}
  \]
  whereas if $i = p^l$ is a power of~$p$, equivalently if
  $ l(i+1) = l + 1$ for $l = l(i)$, then 
  $D_i(G_k)/D_{i+1}(G_k) \cong \langle x^{p^l} \rangle / \langle
  x^{p^{l+1}} \rangle \times \langle y^{p^l} \rangle / \langle
  y^{p^{l+1}} \rangle \times \gamma_i(G_k) / \gamma_{i+1}(G_k)$ so that
  \begin{align*}
    & D_1(G_k) = \langle x,y \rangle D_2(G_k), \\
    & D_p(G_k) = \langle
      x^p,y^p, [y,x,\overset{p-1}{\ldots},x],
      [y,x,\overset{p-2}{\ldots},x,y] \rangle D_{p+1}(G_k), \\
    & D_i(G_k) = \langle x^{p^l}, [y,x,\overset{i-1}{\ldots},x],
      [y,x,\overset{i-2}{\ldots},x,y] \rangle D_{i+1}(G_k)
  \end{align*}
  with
  \[
  D_i(G_k)/D_{i+1}(G_k)\cong
  \begin{cases}
    C_p \times C_p  & \text{if $i=1$, equivalently if $l = 0$,} \\
    C_p \times C_p \times  C_p  \times C_p & \text{if $i=p$,
      equivalently if $l = 1$,} \\
    C_p \times C_p \times C_p & \text{if $i = p^l$ with $2 \le l \le k$.}
  \end{cases}
  \]
  
  In particular, for $p^{k-1}+1 \le i \le p^k$ and thus $l(i)=k$,
  \[
  D_i(G_k) = G_k^{\, p^k} \gamma_i(G_k) = \langle x^{p^k} \rangle \;
  \gamma_i(G_k),
  \]
  so that 
  \[
  \log_p\lvert D_i(G_k) \rvert = \log_p\lvert \gamma_i(G_k) \rvert +1.
  \]
\end{corollary}

\begin{proof}
  For $i \in \mathbb{N}$ write $l(i) = \lceil\log_p i \rceil$. From
  \cite[Thm.~11.2]{DDMS99} and Lemma~\ref{lem:gamma-2-exp-p} we obtain
  $D_i(G_k) = G_k^{\, p^{l(i)}} \gamma_i(G_k)$.  In particular,
  $D_i(G_k) = 1$ for $i > p^k$, by
  Proposition~\ref{pro:lower-central-Gk} and
  Corollary~\ref{cor:lower-p-central-Gk}.

  Now suppose that $1\le i \le p^k$ and put $l = l(i)$.  From
  Lemma~\ref{lem:contained} we observe that
  $G_k^{\, p^l} \cap \gamma_i(G_k) \subseteq \gamma_{p^l}(G_k)$.  If
  $l(i+1) = l$ then  $\gamma_{p^l}(G_k) \subseteq
  \gamma_{i+1}(G_k)$, and hence
  \begin{align*}
    D_i(G_k)/D_{i+1}(G_k) & = G_k^{\, p^l}
                            \gamma_i(G_k)/G_k^{\, p^l} \gamma_{i+1}(G_k)\\
                          & \cong
                            \gamma_i(G_k) / (G_k^{\, p^l} \cap
                            \gamma_i(G_k)) \gamma_{i+1}(G_k)\\  
                          &\cong \gamma_i(G_l) / \gamma_{i+1}(G_l).
  \end{align*}

  Now suppose that $l(i+1) = l+1$, equivalently $i = p^l$.  We observe
  that, modulo $H_k$, the $i$th factor of the dimension subgroup
 series is
  \[
  D_i(G_k) H_k / D_{i+1}(G_k) H_k = \langle x^{p^l} \rangle H_k /
  \langle x^{p^{l+1}} \rangle H_k \cong C_p.
  \]
  Comparing with the overall order of~$G_k$, conveniently implicit in
  Corollary~\ref{cor:lower-p-central-Gk}, we deduce that
  \begin{align*} 
    D_i(G_k)/D_{i+1}(G_k) & = G_k^{\, p^l} \gamma_i(G_k) / G_k^{\, p^{l+1}}
                            \gamma_{i+1}(G_k) \\
                          & = \langle x^{p^l}, y^{p^l}\rangle \gamma_i(G_l) / \langle
                            x^{p^{l+1}}, y^{p^{l+1}} \rangle \gamma_{i+1}(G_l) \\
                          & \cong \langle x^{p^l} \rangle / \langle
                            x^{p^{l+1}} \rangle \times \langle y^{p^l} \rangle / \langle
                            y^{p^{l+1}} \rangle \times \gamma_i(G_l) / \gamma_{i+1}(G_l).
  \end{align*}
  All remaining assertions follow readily from
  Proposition~\ref{pro:lower-central-Gk}.
  %
 \end{proof}



\begin{proposition}\label{pro:Frattini-Gk}
  The Frattini series of $G_k$ has length $k+2$ and satisfies:
  \begin{align*}
    & G_k=\Phi_0(G_k)=\langle x,y\rangle \Phi_1(G_k) \quad
      \text{with}\quad G_k/\Phi_1(G_k) \cong C_p \times C_p, \\
    & \Phi_1(G_k) = \langle x^p,y^p, [y,x], [y,x,x],\ldots,
      [y,x,\overset{p}{\ldots},x], [y,x,y] \rangle\Phi_2(G_k) \\
    & \qquad \text{with}\quad  \Phi_1(G_k)/\Phi_2(G_k) \cong C_p^{\, p+3}, 
  \end{align*}
  and, for $2\le i\le k$, the $i$th term is
  \begin{align*}
    & \Phi_i(G_k)=\langle x^{p^i},[y,x,
      \overset{\nu(i)}{\ldots}, x], [y,x,
      \overset{\nu(i)+1}{\ldots}, x], \ldots, [y,x,
      \overset{\nu(i+1)-1}{\ldots},
      x], \\
    & \qquad  [y,x, \overset{2\nu(i-1)+1}{\ldots}, x,y], [y,x,
      \overset{2\nu(i-1)+3}{\ldots}, x,y],
      \ldots, [y,x, \overset{2\nu(i)-1}{\ldots},
      x,y]\rangle \Phi_{i+1}(G_k) \\
    & \quad \text{with} \quad \Phi_i(G_k)/\Phi_{i+1}(G_k) \cong
      \begin{cases}
        C_p^{\, p^i+p^{i-1}+1} & \text{ for $i \ne k$,}\\
        C_p^{\, p^k+1-(p^{k-1}-1)/(p-1)} & \text{ for $i=k$,}
      \end{cases}
  \end{align*}
  where
  \[
  \nu(j) = \min \big\{ \nicefrac{(p^j-1)}{(p-1)}, p^k \big\} =
  \begin{cases}
    \nicefrac{(p^j-1)}{(p-1)} & \text{for $1 \le j \le k$,} \\
    p^k & \text{for $j = k+1$;}
  \end{cases}
  \]
  lastly,
  \begin{align*}
    & \Phi_{k+1}(G_k) = \langle [y,x,
      \overset{2\nu(k)+1}{\ldots}, x,y], [y,x,
      \overset{2\nu(k)+3}{\ldots}, x,y], \ldots, [y,x,
      \overset{p^k-2}{\ldots}, x,y]\rangle \\
    & \quad
      \text{with} \quad \Phi_{k+1}(G_k)\cong C_p^{\,
      (p^{k+1}-3p^k-p+3)/(2(p-1))}.
  \end{align*}
\end{proposition}

\begin{proof}
  For ease of notation we set $c_1=y$ and, for $i \ge 2$,
  \[
  c_i = [y,x,\overset{i-1}{\ldots},x] \qquad \text{and} \qquad z_i =
  [c_{i-1},y] = [y,x,\overset{i-2}{\ldots},x,y].
  \]
  From Lemma~\ref{lem:gamma-2-exp-p} we observe that
  $c_i^{\, p} = z_i^{\, p} = 1$ for $i \ge 2$; furthermore, the
  elements $z_i \in [H_k,H_k] \subseteq Z_k$ are central in~$G_k$.  We
  claim that
  \begin{equation} \label{equ:comm-ci-cj}
    [c_i,c_j] \equiv z_{i+j}^{\, (-1)^{j-1}} \;\mathrm{mod}\;
      \gamma_{i+j+1}(G_k) \qquad \text{for $i > j \ge 1$.}  
  \end{equation}
  Indeed, $[c_i,c_1] = [c_i,y] = z_{i+1}$, and, modulo
  $\gamma_{i+j+1}(G_k)$, the Hall--Witt identity gives
  \[
  1 \equiv [c_i,c_{j-1},x] [c_{j-1},x,c_i] [x,c_i,c_{j-1}] \equiv [c_j,c_i]
  [c_{i+1},c_{j-1}]^{-1},
  \]
  hence $[c_i,c_j] \equiv [c_{i+1},c_{j-1}]^{-1}$ from which the result
  follows by induction.

  We use the generators specified in the statement of the proposition
  to define an ascending chain
  $1 = L_{k+2} \le L_{k+1} \le \ldots \le L_1 \le L_0 = G_k$ so that
  each $L_i$ is the desired candidate for $\Phi_i(G_k)$.  For
  $1 \le i \le k+1$ we deduce from
  Proposition~\ref{pro:lower-central-Gk} and
  Corollary~\ref{cor:gamma-central-part} that
  \[
  L_i = \langle x^{p^i} \rangle M_i \quad \text{with} \quad M_i =
  \langle c_{\nu(i)+1} \rangle \gamma_{\nu(i)+2}(G_k) C_i
  \trianglelefteq G_k,
  \]
  where
  $C_i = \langle y^{p^i} \rangle \times \langle z_j \mid 2\nu(i-1)+3
  \le j \le p^k \text{ and } j \equiv_2 1 \rangle$
  is central in~$G_k$.  (Note that the factor
  $\langle y^{p^i} \rangle$ vanishes if $i \ge 2$.)  Applying 
    \eqref{equ:commutator-formula-3}, based on
    Proposition~\ref{pro:standard-commutator-id} and
    Lemma~\ref{lem:gamma-2-exp-p}, 
  we see that
  $[x^{p^i},G_k] = [x^{p^i},H_k] \subseteq \gamma_{p^i+1}(G_k)$, hence
  $L_i \trianglelefteq G_k$ for $1 \le i \le k+1$.  Using
  also~\eqref{equ:comm-ci-cj}, we see that the factor groups
  $L_i/L_{i+1}$ are elementary abelian for $0 \le i \le k+1$.  In
  particular, this shows that $\Phi_i(G_k) \subseteq L_i$ for
  $1 \le i \le k+2$.

  Clearly, for each $i \in \{0,\ldots,k+1\}$, the value of
  $\log_p \lvert L_i / L_{i+1} \rvert = d(L_i/L_{i+1})$ is bounded by
  the number of explicit generators used to define $L_i$ modulo
  $L_{i+1}$; these numbers are specified in the statement of the
  proposition and a routine summation shows that they add up to the
  logarithmic order $\log_p \lvert G_k \rvert$, as given in
  Lemma~\ref{lem:order-Gk}.  Therefore each $L_i/L_{i+1}$ has the
  expected rank and it suffices to show that
  $\Phi_i(G_k) \supseteq L_i$ for $1 \le i \le k+1$.

  Let $i \in \{1,\ldots,k+1\}$.  It is enough to show that the
  following elements which generate $L_i$ as a normal subgroup belong
  to~$\Phi_i(G_k)$:
  \[
  x^{p^i}, \quad c_{\nu(i)+1}, \qquad \text{and} \quad z_j \quad
  \text{for} \quad \text{$2\nu(i-1)+3 \le j \le p^k$ with
    $j \equiv_2 1$.}
  \]
  Clearly, $x^{p^i} \in \Phi_i(G_k)$ and, applying 
    \eqref{equ:commutator-formula-3}, based on
    Proposition~\ref{pro:standard-commutator-id} and
    Lemma~\ref{lem:gamma-2-exp-p}, 
  we see by induction on $i$ that
  \[
  c_{\nu(i)+1} = [y,x,\overset{\nu(i)}{\ldots},x] \equiv_{\Phi_i(G_k)}
  [y,x,x^p,\ldots,x^{p^{i-1}}] \equiv_{\Phi_i(G_k)} 1.
  \]
  Now let $2\nu(i-1)+3 \le j \le p^k$ with $j \equiv_2 1$.  By
  Corollary~\ref{cor:gamma-central-part} and reverse induction on $j$
  it suffices to show that $z_j$ is contained in $\Phi_i(G_k)$ modulo
  $\gamma_{j+1}(G_k)$.  This follows from \eqref{equ:comm-ci-cj}
  and the fact that
  $c_{\nu(i-1)+1}, c_{j-\nu(i-1)-1} \in \Phi_{i-1}(G_k)$ by induction
  on~$i$.
\end{proof}

Using Corollary~\ref{cor:general-normal-hspec}, we can now complete
the proof of Theorem~\ref{thm:main-thm}: it suffices to compute
$\hdim^{\mathcal{S}}_G(Z)$ and $\hdim^{\mathcal{S}}_G(H)$ for the
standard filtration series
$\mathcal{S} \in \{ \mathcal{L}, \mathcal{D}, \mathcal{F} \}$.

Corollary~\ref{cor:lower-p-central-Gk} implies
\begin{align}
  \hdim^{\mathcal{L}}_G(Z) & = \varliminf_{i \to \infty}
                             \frac{\log_p \lvert Z P_i(G) : P_i(G) \rvert}{\log_p \lvert G : P_i(G)
                             \rvert} = \lim_{i \to \infty}
                             \frac{\nicefrac{i}{2}}{\nicefrac{5i}{2}}
                             = \nicefrac{1}{5}, %
                             \label{equ:xi-for-L} \\
  \hdim^{\mathcal{L}}_G(H) & = \varliminf_{i \to \infty}
                             \frac{\log_p \lvert H P_i(G) : P_i(G) \rvert}{\log_p \lvert G : P_i(G)
                             \rvert} = \lim_{i \to \infty}
                             \frac{\nicefrac{3i}{2}}{\nicefrac{5i}{2}}
                             = \nicefrac{3}{5}. %
                             \label{equ:eta-for-L}
\end{align}

Corollary~\ref{cor:dim-subgroup} implies
\begin{align}
  \hdim^{\mathcal{D}}_G(Z) & = \varliminf_{i \to \infty}
                             \frac{\log_p \lvert Z D_i(G) : D_i(G) \rvert}{\log_p \lvert G : D_i(G)
                             \rvert} = \lim_{i \to \infty}
                             \frac{\nicefrac{i}{2}}{\nicefrac{3i}{2}}
                             = \nicefrac{1}{3}, %
                             \label{equ:xi-for-D} \\
  \hdim^{\mathcal{D}}_G(H) & = \varliminf_{i \to \infty}
                             \frac{\log_p \lvert H D_i(G) : D_i(G) \rvert}{\log_p \lvert G : D_i(G)
                             \rvert} = \lim_{i \to \infty}
                             \frac{\nicefrac{3i}{2}}{\nicefrac{3i}{2}}
                             = 1. \notag
\end{align}

Lastly, Proposition~\ref{pro:Frattini-Gk}
implies
\begin{align}
  \hdim^{\mathcal{F}}_G(Z) & = \varliminf_{i \to \infty}
                             \frac{\log_p \lvert Z \Phi_i(G) :
                             \Phi_i(G) \rvert}{\log_p \lvert G :
                             \Phi_i(G) \rvert} = \lim_{i \to \infty}
                             \frac{\sum_{j=1}^{i-1} p^{j-1}}{\sum_{j=1}^{i-1} (p^j 
                             + p^{j-1} +1)} = \nicefrac{1}{p+1}, %
\label{equ:xi-for-F} \\
  \hdim^{\mathcal{F}}_G(H) & = \varliminf_{i \to \infty}
                             \frac{\log_p \lvert H \Phi_i(G) :
                             \Phi_i(G) \rvert}{\log_p \lvert G :
                             \Phi_i(G) \rvert} = \lim_{i \to \infty}
                             \frac{\sum_{j=1}^{i-1} (p^j 
                             + p^{j-1})}{\sum_{j=1}^{i-1} (p^j 
                             + p^{j-1} +1)} = 1. \notag
\end{align}

\begin{remark} \label{Z-H-always_strong}
  From \eqref{equ:xi-for-P}, \eqref{equ:xi-for-L},
  \eqref{equ:eta-for-L}, \eqref{equ:xi-for-D}, \eqref{equ:xi-for-F}
  and the fact that subgroups of Hausdorff dimension~$1$ automatically
  have strong Hausdorff dimension we conclude that $Z$ and $H$ have
  strong Hausdorff dimension in~$G$ with respect to all standard
  filtration series $\mathcal{P}$, $\mathcal{D}$, $\mathcal{F}$ and
  $\mathcal{L}$.
\end{remark}


\section{The entire Hausdorff spectra of $G$ with respect to the
  standard filtration series} \label{sec:entire-spectrum}

We continue to use the notation set up in
Section~\ref{sec:presentation} to study and determine the entire
Hausdorff spectra of the pro-$p$ group~$G$, with respect to the
standard filtration series
$\mathcal{P}, \mathcal{D}, \mathcal{F}, \mathcal{L}$.

\begin{proof}[Proof of Theorem~\ref{thm:entire-spectrum}]
  As in Sections~\ref{sec:prelim} and~\ref{sec:presentation}, we write
  $W = G/Z \cong C_p \mathrel{\hat{\wr}} \mathbb{Z}_p$, and we denote
  by $\pi \colon G \rightarrow W$ the canonical projection with
  $\ker \pi = Z$.

  \medskip

  First suppose that $\mathcal{S}$ is one of the filtration series
  $\mathcal{P}, \mathcal{D}, \mathcal{F}$ on~$G$.  By
  Remark~\ref{Z-H-always_strong}, the group $H$ has strong Hausdorff
  dimension~$1$ in~$G$ with respect to~$\mathcal{S}$. As every finitely
  generated subgroup of $H$ is finite, it follows
  from~\cite[Thm.~5.4]{KlThZu19} that
  $\hspec^{\mathcal{S}}(G) = [0,1]$.

 
  It remains to pin down the Hausdorff spectrum of $G$ with respect to
  the lower $p$-series $\mathcal{L} \colon P_i(G)$,
  $i \in \mathbb{N}$, on~$G$.  By Remark~\ref{Z-H-always_strong}, the
  normal subgroups $Z, H \trianglelefteq_\mathrm{c} G$ have strong
  Hausdorff dimensions $\hdim^{\mathcal{L}}_G(Z) = \nicefrac{1}{5}$
  and $\hdim^{\mathcal{L}}_G(H) = \nicefrac{3}{5}$.  From
  Corollary~\ref{cor:abelian-section-interval},
  Lemma~\ref{lem:hdims-from-ses} and Corollary~\ref{cor:spectrum-W-L}
  we deduce that $\hspec^{\mathcal{L}}(G)$ contains
  \[
  S = [0,\nicefrac{3}{5}] \cup
  \{\nicefrac{3}{5} + \nicefrac{2m}{5p^n}\mid m, n \in\mathbb{N}_0
  \text{ with } \nicefrac{p^n}{2} < m\le p^n\}.
  \]
  Thus it suffices to show that
  \begin{equation} \label{equ:3/5-4/5} 
    (\nicefrac{3}{5},\nicefrac{4}{5}) \subseteq \hspec^{\mathcal{L}}(G)
    \subseteq  (\nicefrac{3}{5},\nicefrac{4}{5}) \cup S.
   \end{equation}
   
   First we prove the second inclusion.  Let $K \le_\mathrm{c} G$ be
   any closed subgroup with
   $\hdim^\mathcal{L}_G(K) > \nicefrac{3}{5}$.  In particular, this
   implies $K \not \subseteq H$ and hence $KH \le_\mathrm{o} G$.

   We denote by $\mathcal{L} \vert_H$ and $\mathcal{L} \vert_{H\pi}$
   the filtration series induced by $\mathcal{L}$ on $H$, via
   intersection, and on $H \pi = HZ/Z$, via subsequent reduction
   modulo~$Z$.  We write $\mathcal{L}$ for the filtration series
   $\mathcal{L} \vert_W $ induced on~$W = G/Z$, as it coincides with
   the lower $p$-series of the quotient group.  Using
   Corollary~\ref{cor:spectrum-W-L} and
   Lemma~\ref{lem:hdims-from-ses}, we see that $(K \cap H) \pi$ has
   strong Hausdorff dimension
  \[
  \alpha = \hdim^{\mathcal{L} \vert_{H\pi}}_{H \pi}((K \cap H)\pi) = 2
  \hdim^\mathcal{L}_W(K\pi) - 1 \in [0,1]
  \]
  in $H \pi$ with respect to~$\mathcal{L} \vert_{H\pi}$.  Applying
  Lemma~\ref{lem:hdims-from-ses} twice, we deduce that
  \begin{equation} \label{equ:twice}
    \begin{split}
      \hdim^\mathcal{L}_G(K) & = \tfrac{2}{5} + \tfrac{3}{5}
      \hdim^{\mathcal{L} \vert_H}_H(K \cap H)\\
      & \in \tfrac{2}{5} + \tfrac{3}{5} \left( \tfrac{2}{3}
        \hdim^{\mathcal{L} \vert_{H\pi}}_{H\pi} \big( (K \cap H)\pi
        \big) + [0,\nicefrac{1}{3}] \right) \\
      & = \tfrac{2}{5} (1 +\alpha) + [0,\nicefrac{1}{5}].
    \end{split}
  \end{equation}
  For $\alpha < \nicefrac{1}{2}$ we obtain
  $\hdim^\mathcal{L}_G(K) < \nicefrac{4}{5}$ and there is nothing
  further to prove.  Now suppose that $\alpha \ge \nicefrac{1}{2}$.
  It suffices to show that $K \cap Z \le_\mathrm{o} Z$ and hence
  $\hdim^\mathcal{L}_G(K \cap Z) = \nicefrac{1}{5}$: with this extra
  information we can refine the analysis in~\eqref{equ:twice} and use
  Corollary~\ref{cor:spectrum-W-L} once more to deduce that
  \[
  \hdim^\mathcal{L}_G(K) = \tfrac{2}{5} (1 +\alpha) + \tfrac{1}{5} =
  \tfrac{4}{5} \hdim^\mathcal{L}_W(K \pi) + \tfrac{1}{5} \in
  S.
  \]
  Let us prove that $K \cap Z \le_\mathrm{o} Z$.  As
  $KH \le_\mathrm{o} G$, we have $KH = \langle x^{p^n} \rangle H$,
  where $n = \log_p \lvert G : KH \rvert \in \mathbb{N}_0$.  Using
  Lemma~\ref{lem:hdims-from-ses}, we deduce from
  $\alpha \ge \nicefrac{1}{2}$ that
  \begin{equation} \label{equ:more-than-half} \hdim^\mathcal{L}_W((K
    \cap H) \pi) \ge \nicefrac{1}{4} = \tfrac{1}{2} \hdim^\mathcal{L}_W(H \pi).
  \end{equation}
  At this point it is useful to recall our analysis
  of~$\hspec^\mathcal{L}(W)$ in the proof of
  Theorem~\ref{thm:finitely-generated-spectrum} and also the
  computations carried out in the proof of
  Proposition~\ref{pro:Frattini-Gk}, involving the elements
  $c_i = [y,x,\overset{i-1}{\ldots},x]$ and~$z_i = [c_{i-1},y]$.  In
  particular, for $i \in \mathbb{N}$ with $i \ge 3$ we have
  \[
  (P_i(G) \cap H)\pi  = \langle c_j \mid j \ge i \rangle \pi
  \quad \text{and} \quad
  P_i(G) \cap Z = \langle z_j \mid  j \ge i \text{ and } j \equiv_2 1  \rangle;
  \]
  compare Corollary~\ref{cor:lower-p-central-Gk}.
  From~\eqref{equ:more-than-half} and the proof of
  Theorem~\ref{thm:finitely-generated-spectrum} we deduce that,
  subject to replacing $K$ by a suitable open subgroup
  $\widetilde{K} = K \cap \langle x^{p^{\widetilde{n}}} \rangle H$
  with $\widetilde{n} \ge n$ if necessary, we find $m \ge (p^n+1)/2$
  and $a_1, \ldots, a_m \in K \cap H$ so that
  \[
  (K \cap H) M / M = \langle a_1, \ldots, a_m \rangle M /M \cong
  C_p^{\, m}, \qquad \text{where $M = (P_{p^n+1}(G) \cap H) Z$,}
  \] 
  and the numbers
  \[
  d(j) = \max \{ i \in \mathbb{N} \mid a_j \in (P_i(G) \cap H)Z \},
  \quad 1 \le j \le m,
  \]
  form a strictly increasing sequence
  $1 \le d(1) < \ldots < d(m) < p^n$.  Commuting $a_1,\ldots,a_m$
  repeatedly with $x^{p^n}$, we see as in the proof of
  Theorem~\ref{thm:finitely-generated-spectrum} that
  \[
    \{ d(1), \ldots, d(m) \} + p^n \mathbb{N}_0
    \subseteq \{ i \in \mathbb{N} \mid \exists g \in K \cap H : g
    \equiv_{P_{i+1}(G) Z} c_i \}.
  \]
  For every $k \in \mathbb{N}$ with $k > p^n$ and $k \equiv_2 1$, the
  pigeonhole principle (Dirichlet's `Schubfachprinzip') yields
  $i,j \in \mathbb{N}$ with $i>j \ge 1$ and $i+j = k$, and we find
  $g_i, g_j \in K \cap H$ with $g_i \equiv_{P_{i+1}(G) Z} c_i$ and
  $g_j \equiv_{P_{j+1}(G) Z} c_j$ so that~\eqref{equ:comm-ci-cj} gives
  \[
  z_k \equiv_{P_{k+1}(G)} [c_i,c_j]^{(-1)^{j-1}} \equiv_{P_{k+1}(G)}
  [g_i,g_j]^{(-1)^{j-1}} \in K \cap Z.
  \]
  But this implies
  $K \cap Z \supseteq \langle z_j \mid j>p^n \text{ and } j \equiv_2 1
  \rangle = P_{p^n+1}(G) \cap Z$
  and thus $K \cap Z \le_\mathrm{o} Z$.  This concludes the proof
  of the second inclusion in~\eqref{equ:3/5-4/5}.

  Finally we prove the first inclusion in~\eqref{equ:3/5-4/5}.  Let
  $\xi  \in (\nicefrac{2}{5},\nicefrac{4}{5})$.  Choose $m,n \in
  \mathbb{N}$ such that $1 \le m < \nicefrac{p^n}{2}$ and
  \[
  \tfrac{1}{5} \big( 2 + \nicefrac{(4m -1)}{p^n} \big) \le \xi \le
  \tfrac{1}{5} \big( 3 + \nicefrac{2 m}{p^n} \big). 
  \] 
  Consider the group
  $K = \langle x^{p^n}, y_0, y_1, \ldots, y_{m-1} \rangle$.  Using the
  proof of Theorem~\ref{thm:finitely-generated-spectrum} and
  Lemma~\ref{lem:hdims-from-ses}, we show below that $K$ has Hausdorff
  dimension
  \begin{equation} \label{equ:K-strong-hdim}
    \begin{split}
      \qquad \hdim^\mathcal{L}_G(K) & = \tfrac{4}{5}
      \hdim^\mathcal{L}_W (K
      \pi) + \tfrac{1}{5} \hdim^{\mathcal{L} \vert_Z}_Z(K \cap Z) \\
      & = \big( \tfrac{2}{5} + \tfrac{2}{5} \tfrac{m}{p^n} \big) +
      \tfrac{1}{5} \tfrac{2m-1}{p^n} \\
      & = \tfrac{1}{5} \big( 2 + \nicefrac{(4m - 1)}{p^n} \big).
    \end{split}
  \end{equation}
  In a similar, but much more straightforward way, we see that $ZK$
  has strong Hausdorff dimension
  \[
  \hdim^\mathcal{L}_G(ZK) = \big( \tfrac{2}{5} + \tfrac{2}{5} \tfrac{m}{p^n} \big) +
      \tfrac{1}{5} = \tfrac{1}{5} \big( 3
  + \nicefrac{2 m}{p^n} \big).
  \]
  An application of~\cite[Thm.~5.4]{KlThZu19} yields
  $L \le_\mathrm{c} G$ with $K \le L \le ZK$ such that
  $\hdim^\mathcal{L}_G(L) = \xi$.

  The key to~\eqref{equ:K-strong-hdim} consists in showing that
  \begin{equation}\label{equ:the-key}
    \varliminf_{i \to \infty} \frac{\log_p \lvert K P_i(G) \cap Z : P_i(G) \cap Z
      \rvert}{\log_p \lvert Z : P_i(G) \cap Z \rvert} =
    \hdim_Z^{\mathcal{L} \vert_Z}(K \cap Z) = (2m-1)/p^n.
  \end{equation}
  First we examine the lower limit on the left-hand side, restricting
  to indices of the form $i = p^k +1$, $k \in \mathbb{N}$.  Let
  $i = p^k+1$, where $k \ge n$.  Recall that
    $G_k = G / \langle x^{p^{k+1}}, [x^{p^k},y] \rangle^G$ 
  and consider the canonical projection $\varrho_k \colon G \to G_k$,
  $g \mapsto \overline{g}$.  As before, we write $H_k = H \varrho_k$. 
    Furthermore, we observe that
    $Z_k = \langle \overline{x}^{p^k} \rangle Z \varrho_k$ with
    $\lvert Z_k : Z \varrho_k \rvert = p$. 
  By Corollary~\ref{cor:lower-p-central-Gk}, we have
  \[
  \lvert H_k : H_k \cap \underbrace{P_i(G_k)}_{=1} \rvert = \lvert H_k
  \rvert = \lvert   H : H \cap P_i(G) \rvert
  \]
  and hence
  \[
  \frac{\log_p \lvert K P_i(G) \cap Z : P_i(G) \cap Z \rvert}{\log_p
    \lvert Z : P_i(G) \cap Z \rvert} = \frac{\log_p \lvert K \varrho_k
    \cap  Z \varrho_k \rvert}{\log_p \lvert
     Z \varrho_k \rvert}.
  \]
  Observe that
  \[
  K \varrho_k \cap H_k = \langle \overline{y_j} \mid 0 \le j < p^k \text{ with } j
  \equiv_{p^n} 0,1, \ldots, m-1 \rangle.
  \]
  From Lemma~\ref{lem:order-Gk} we 
    see that $Z \varrho_k \cong C_p^{\, (p^k+1)/2}$ 
  and further we deduce that
  \begin{align*}
    K & \varrho_k \cap Z \varrho_k \\
      & =  \langle \{ \overline{y}^p \} 
        \cup \{ [ \overline{y_0}, \overline{ y_j}] \mid 0
        \le j < p^k, \; j \equiv_{p^n} 0, \pm 1,\ldots, \pm (m-1), \;  j \equiv_2 0 \}
        \rangle \\
      & \cong  C_p^{\, ((2m-1) p^{k-n} + 1)/2}.
  \end{align*}
  This yields
  \begin{align*}
    \varliminf_{i \to \infty} \frac{\log_p \lvert K P_i(G) \cap Z :
    P_i(G) \cap Z \rvert}{\log_p \lvert Z : P_i(G) \cap Z \rvert}
    & \le \varliminf_{k \to \infty} \frac{\log_p \lvert K \varrho_k \cap
     Z \varrho_k \rvert}{\log_p \lvert
       Z \varrho_k \rvert}\\
    & = \lim_{k \to \infty}  \frac{(2m-1) p^{k-n}
      +1}{p^k+1} = (2m-1)/p^n.
  \end{align*}
  In order to establish~\eqref{equ:the-key} it now suffices to prove
  that
  \begin{equation} \label{equ:last-reduction}
  \varliminf_{i \to \infty} \frac{\log_p \lvert (K \cap Z) (P_i(G)
    \cap Z) : P_i(G) \cap Z \rvert}{\log_p \lvert Z : P_i(G) \cap Z
    \rvert} \ge (2m-1)/p^n.
  \end{equation}
  Our analysis above yields
  \[
  K \cap Z = \langle \{ y^p\} \cup \{ [y_0, y_j] \mid j \in \mathbb{N}
  \text{ with } j \equiv_{p^n} 
  0, \pm 1,\ldots,\pm (m-1) \} \rangle.
  \]
  Setting
  \[
  L= \langle y_j \mid j \in \mathbb{N}_0 \text{ with } j \equiv_{p^n} 
  0, \pm 1, \ldots, \pm (m-1) \rangle Z,
  \]
  and recalling the notation $c_1 = y = y_0$, we conclude that
  \[
  K \cap Z \supseteq \{ [g,c_1] \mid g \in L \}.
  \] 

  Next we consider the set
  \[
  D = \{ j \in \mathbb{N} \mid \exists g \in L :  g \equiv_{P_{j+1}(G) Z} c_j \}.
  \]
  Each element $y_j$ can be written (modulo $Z$) as a product
  \[
  y_j \equiv_Z \prod_{k=0}^j c_{k+1}^{\, \beta(j,k)} \qquad
  \text{where $\beta(j,k) = \tbinom{j}{k}$,}
  \]
  using the elements $c_i = [y,x,\overset{i-1}{\ldots},x]$ introduced
  in the proof of Proposition~\ref{pro:Frattini-Gk}.  In this product
  decomposition, the exponents should be read modulo~$p$, and the
  elementary identity $(1+t)^{j + p^n} = (1+t)^j (1+t^{p^n})$ in
  $\mathbb{F}_p[\![t]\!]$ translates to
  \[
  y_j^{\, -1} y_{j + p^n} = y^{-x^j} y^{x^{j + p^n}} \equiv_Z \prod_{k=0}^j c_{k+1+p^n}^{\,
    \beta(j,k)}  \qquad \text{for all $j \in \mathbb{N}$;}
  \]
  compare~\eqref{equ:Wk-explicit}.  Inductively, we obtain
  \[
  D = D_0 + p^n \mathbb{N}_0 \qquad \text{for $D_0 = D \cap
    \{1,\ldots,p^n\}$}.
  \]
  Observe that $\lvert D_0 \rvert = 2m-1$ and that, for each
  $k \in \mathbb{N}_0$, the set
  $(2k p^n + D_0) \cup ((2k+1)p^n + D_0)$ consists of $2m-1$ odd and
  $2m-1$ even numbers.

  For each $j \in D$ with $j \equiv_2 0$ there exists $g_j \in L$ with
  $g_j \equiv_{P_{j+1}(G) Z} c_j$ and we deduce that
  \[
  z_{j+1} = [c_j,c_1] \equiv_{P_{j+2}(G)}
  [g_j,c_1] \in K \cap Z.
  \]
  For $i = 2p^n q + r \in \mathbb{N}$, where $q,r \in \mathbb{N}_0$
  with $0 \le r < 2p^n$, the count
  \[
  \lvert \{ j \in D \mid j \equiv_2 0 \text{ and } j < i-1 \} \rvert
  \geq q (2m-1) -1
  \]
  yields
  \[
  \log_p \lvert (K \cap Z) (P_i(G) \cap Z) : P_i(G) \cap Z \rvert
  \ge q (2m-1) -1.
  \]
  From Corollary~\ref{cor:lower-p-central-Gk} we observe that, for
  $i \ge 3$,
  \[
  \log_p \lvert Z : P_i(G) \cap Z \rvert = \lfloor \nicefrac{i}{2} 
  \rfloor \le
  q p^n + p^n. 
  \]
  These estimates show that \eqref{equ:last-reduction} holds.
\end{proof}


\appendix

\section{The case $p=2$} \label{sec:case-p-2}

When $p$ is even, Theorems~\ref{thm:main-thm} and
\ref{thm:entire-spectrum}, and all the results of
Sections~\ref{sec:prelim} and \ref{sec:general-description}, hold with
corresponding proofs.  The structural results of
Sections~\ref{sec:p-power-series} and \ref{sec:other-series} however
are slightly different and we now sketch these differences below; for
complete details, we refer the reader to the supplement
\cite{supplement}.

Firstly, for $p=2$,
\begin{equation}\label{equ:pres-for-G-k-p2}
  \begin{split}
    G_k = F/N_k & \cong \langle x,y \mid \phantom{[}x^{2^{k+1}},\;
    y^{4},\; [x^{2^k},y],\; [y^2,x];\\
    & \qquad \quad [y_0,y_i]^2,\; [y_0,y_i,x],\; [y_0,y_i,y] \quad
    \text{for $1\le i\le 2^{k-1}$} \rangle
  \end{split}
\end{equation}
for $k \in \mathbb{N}$, and 
\begin{equation}\label{equ:pres-for-G-p2}
  G \cong \langle x,y \mid y^{4},\; [y^2,x];\; [y_0,y_i]^2,\;
  [y_0,y_i,x],\; [y_0,y_i,y] \quad \text{for $i \in \mathbb{N}$} \rangle
\end{equation}
is a presentation of $G$ as a pro-$2$ group. 

Next, we have $\log_2|G_k|=2^k+2^{k-1}+k+2$ and the exponent of
$\gamma_2(G_k)$ is 4. With regards to Lemma~\ref{lem:w-elements} , the
elements
\[
w = y_{2^k-1}\cdots y_1y_0 \qquad \text{and} \qquad
[w,x]=[w,y]=[y_0,y_{2^{k-1}}]
\]
are of order $2$ in $G_k$ and lie in $G_k^{\, 2^k}$.  In particular
the subgroup $\langle x^{2^k},w,[w,x]\rangle$ is isomorphic to
$C_2\times C_2\times C_2$ and lies in $G_k^{\, 2^k}$.  Hence, for
$k\ge 2$,
\[
G_k^{\, 2^k} =\langle x^{2^k}, w, [w,x]\rangle \cong C_2 \times C_2
\times C_2, \qquad \log_2 \vert G_k : G_k^{\, 2^k} \rvert = \log_2
\lvert G_k \rvert - 3
\]
and
\begin{equation*}
  \begin{split}
    G_k / G_k^{\, 2^k} & \cong \langle x,y \mid x^{2^{k}},\; y^{4},\;
    [y^2,x],\; w(x,y),\; [y_0,y_{2^{k-1}}];\\
    & \qquad \quad [y_0,y_i]^2,\; [y_0,y_i,x],\; [y_0,y_i,y] \quad
    \textup{for } 1\le i< 2^{k-1} \rangle.
  \end{split}
\end{equation*}
  
Lemma~\ref{lem:contained} is slightly different; here the group $G_k$
satisfies
$G_k^{\, 2} \subseteq \langle x^{2}, y^2\rangle\gamma_{2}(G_k)$ and
\begin{equation*}
  G_k^{\, 2^j} \subseteq \langle
  x^{2^j},[y,x,\overset{2^{j}-3}\ldots,x,y]\rangle\gamma_{2^{j}}(G_k)
  \subseteq \langle x^{2^j}\rangle\gamma_{2^{j}-1}(G_k)\quad \text{
    for }j\ge 2. 
\end{equation*}
The proof is similar, but one needs the fact
\[ 
[y,x,\overset{i}\ldots,x]^2\in [H_k,H_k]\cap
\gamma_{2i+1}(G_k),\quad \text{ for } i\ge 1,
\]
which is proved by induction, using
\[ 
[[y,x,\overset{i-1}\ldots, x],x]^2= [[y,x,\overset{i-1}\ldots, x]^x,
[y,x,\overset{i-1}\ldots, x]^{-1}] \qquad \text{for $i\ge 2$.}
\]
Furthermore,
$[y,x,\overset{i}\ldots,x]^2\equiv [y,x,\overset{2i-1}\ldots,x,y]$
modulo $\gamma_{2i+2}(G_k)$.
  
\smallskip
  
The group $G_k$ is nilpotent of class $2^k+1$; its lower central
series satisfies
\[
G_k = \gamma_1(G_k) = \langle x,y \rangle \; \gamma_2(G_k) \quad
\text{with} \quad G_k/\gamma_2(G_k) \cong C_{2^{k+1}} \times C_{4}
\]
and, for $1 \le i \le 2^{k-1}$,
\begin{align*}
  \gamma_{2i}(G_k) & = \langle [y,x,\overset{2i-1}{\ldots},x] \rangle
                     \; \gamma_{2i+1}(G_k), \\
  \gamma_{2i+1}(G_k) & = 
                       \begin{cases}
                         \langle [y,x,\overset{2i}{\ldots},x],
                         [y,x,\overset{2i-1}{\ldots},x,y] \rangle \;
                         \gamma_{2i+2}(G_k) & \text{ for }i\ne 2^{k-1}\\
                         \langle [y,x,\overset{2i}{\ldots},x] \rangle
                         \; \gamma_{2i+2}(G_k) & \text{ for }i=
                         2^{k-1}
                       \end{cases}
\end{align*} 
with
\[
\gamma_{2i}(G_k)/\gamma_{2i+1}(G_k) \cong C_2 \quad \text{and} \quad
\gamma_{2i+1}(G_k)/\gamma_{2i+2}(G_k) \cong
\begin{cases}
  C_2 \times C_2 & \text{ for }i\ne 2^{k-1}\\
  C_2 &\text{ for }i=2^{k-1}.
\end{cases}
\]

The proof of the above is similar to that for the odd prime case,
however here one takes
\[
j_0 =
\begin{cases}
  2^{k-1} - \frac{m}{2} & \text{if $m \equiv_4
    0$,} \\
  2^{k-1}+1 - \frac{m}{2} & \text{if $m \equiv_4 2$.}
\end{cases}
\]
For the $m\equiv_4 0$ case, noting that
$e_{2^{k-1}}=[w,x]\in \gamma_{2^k+1}(G_k)$, we have
$b_{j_0,m} \equiv b_{j_0,m+1}$ modulo $\gamma_{m+1}(G_k)$. The
$m\equiv_4 2$ case is similar.
      
\smallskip
 
The lower $2$-series of $G_k$ has length $2^k+1$ and satisfies the
corresponding form, based on the lower central series of $G_k$ above.
      
\smallskip

The dimension subgroup series of $G_k$ has length~$2^k+2$.  For
$1\le i \le 2^k+2$, the $i$th term is
$D_i(G_k) =G_k^{\, 2^{l(i)}} \gamma_{\lceil
  \nicefrac{i}{2}\rceil}(G_k)^2\gamma_i(G_k)$,
where $l(i) = \lceil\log_2 i \rceil$.

Furthermore, if $i$ is not a power of $2$, equivalently if
$l(i+1)=l(i)$, then
$D_i(G_k)/D_{i+1}(G_k) \cong \gamma_{\lceil
  \nicefrac{i}{2}\rceil}(G_k)^2\gamma_i(G_k) /\gamma_{\lceil
  \nicefrac{i+1}{2}\rceil}(G_k)^2 \gamma_{i+1}(G_k)$ so that
\[
D_i(G_k) =
\begin{cases}
  \langle [y,x,\overset{i-1}{\ldots},x] \rangle D_{i+1}(G_k) &
  \text{if $i\equiv_2 1$,}  \\
  \langle [y,x,\overset{i-3}{\ldots},x,y], [y,x,\overset{i-1}\ldots,x]
  \rangle D_{i+1}(G_k) & \text{if $i\equiv_2 0$,}
\end{cases}
\]
with
\[
D_i(G_k)/D_{i+1}(G_k) \cong
\begin{cases}
  C_2 &
  \text{if $i\equiv_2 1$   and $i<2^k$,}\\
  C_2\times C_2 &
  \text{if $i\equiv_2 0$  and $i<2^k$,}\\
  1 &
  \text{if $i=2^k+1$,}\\
  C_2 &
  \text{if $i=2^k+2$.}\\
\end{cases}
\]
whereas if $i = 2^l$ is a power of~$2$, equivalently if
$ l(i+1) = l + 1$ for $l = l(i)$, then
$D_i(G_k)/D_{i+1}(G_k) \cong \langle x^{2^l} \rangle / \langle
x^{2^{l+1}} \rangle \times \langle y^{2^l} \rangle / \langle
y^{2^{l+1}} \rangle \times \langle
[y,x,\overset{i-3}\ldots,x,y]\rangle\gamma_i(G_k) / \gamma_{i+1}(G_k)$
so that
\begin{align*}
  & D_1(G_k) = \langle x,y \rangle D_2(G_k), \\
  & D_2(G_k) = \langle
    x^2,y^2, [y,x] \rangle D_{3}(G_k), \\
  & D_i(G_k) = \langle x^{2^l}, [y,x,\overset{i-3}\ldots,x,y],
    [y,x,\overset{i-1}{\ldots},x] \rangle D_{i+1}(G_k) 
\end{align*}
with
\[
D_i(G_k)/D_{i+1}(G_k)\cong
\begin{cases}
  C_2 \times C_2  & \text{if $i=1$, equivalently if $l = 0$,} \\
  C_2 \times C_2 \times C_2 & \text{if $i=2$,
    equivalently if $l = 1$,} \\
  C_2 \times C_2 \times C_2 & \text{if $i = 2^l$ with
    $2 \le l \le k$.}
\end{cases}
\]
  
In particular, for $2^{k-1}+1 \le i \le 2^k$ and thus $l(i)=k$,
\[
D_i(G_k) = G_k^{\, 2^k} \gamma_i(G_k) = \langle x^{2^k}, [y,x,
\overset{2^k-3}{\ldots}, x,y] \rangle \; \gamma_i(G_k),
\]
so that
\[
\log_2\lvert D_i(G_k) \rvert = \log_2\lvert \gamma_i(G_k) \rvert +1.
\]
      
\smallskip 
      
Lastly, the Frattini series of $G_k$ has the corresponding form,
though it has length $k+1$.


\end{document}